\documentclass[reqno]{amsart}
\usepackage{amssymb}
\usepackage{amsmath}
\usepackage[mathscr]{euscript}
\usepackage{stmaryrd}
\usepackage[small]{caption}
\usepackage{psfrag}
\usepackage{graphicx}
\usepackage{color}
\usepackage[pdftex,breaklinks,colorlinks,
citecolor=blue,
urlcolor=blue,
pdftitle={Hydrodynamic behavior for boundary driven stirring dynamics},
pdfauthor={C. Erignoux and C. Landim}]{hyperref}
\usepackage[dvipsnames]{xcolor}

\usepackage{tikz}
\usetikzlibrary{backgrounds}
\usetikzlibrary{patterns,fadings}
\usetikzlibrary{arrows,decorations.pathmorphing}
\usetikzlibrary{decorations.pathreplacing}
\usetikzlibrary{decorations}
\usetikzlibrary{calc}
\definecolor{light-gray}{gray}{0.95}
\usepackage{float}
\def\centerarc[#1](#2)(#3:#4:#5){\draw[#1] ($(#2)+({#5*cos(#3)},{#5*sin(#3)})$) arc (#3:#4:#5);}

\allowdisplaybreaks 

\makeatletter
\@addtoreset{equation}{section}
\makeatother

\newtheorem{theorem}{Theorem}[section]
\newtheorem{lemma}[theorem]{Lemma}
\newtheorem{proposition}[theorem]{Proposition}

\newtheorem{remark}[theorem]{Remark}

\newcommand{\mf}[1]{{\mathfrak #1}}

\newcommand{\bb}[1]{{\mathbb #1}}
\newcommand{\bs}[1]{{\boldsymbol #1}}
\newcommand{\ms}[1]{{\mathscr #1}}

\newcounter{as}[section]

\newcommand{\<}{\langle}
\renewcommand{\>}{\rangle}

\newcommand{\R}{{\mathbb R}}

\newcommand{\abs}[1]{\left|#1 \right|}

\newcommand{\pa}[1]{\left(#1 \right)}

\definecolor{dkgreen}{rgb}{0,0.6,0}
\definecolor{gray}{rgb}{0.5,0.5,0.5}
\definecolor{header}{gray}{0.3}

\title[boundary driven exclusion processes]{Stationary states of
  boundary driven exclusion processes with nonreversible boundary
  dynamics}

\author{C. Erignoux, C. Landim, T. Xu}

\address{\noindent Equipe PARADYSE, Bureau B211
Centre INRIA Lille Nord-Europe
Park Plaza, Parc scientifique de la Haute-Borne, 40 Avenue Halley B\^atiment B, 59650 Villeneuve-d'Ascq
France
  \newline e-mail: \rm \texttt{clement.erignoux@inria.fr} }

\address{\noindent IMPA, Estrada Dona Castorina 110, CEP 22460 Rio de
  Janeiro, Brasil and CNRS UMR 6085, Universit\'e de Rouen, France.
  \newline e-mail: \rm \texttt{landim@impa.br} }

\address{\noindent IMPA, Estrada Dona Castorina 110, CEP 22460 Rio de
  Janeiro, Brasil 
  \newline e-mail: \rm \texttt{tcxu@impa.br} }

\begin{document}

\begin{abstract}
  We prove a law of large numbers for the empirical density of
  one-dimensional, boundary driven, symmetric exclusion processes with
  different types of non-reversible dynamics at the boundary. The
  proofs rely on duality techniques.
\end{abstract}

\maketitle

\section{Introduction}
\label{sec0}

This article provides partial answers to a question raised to us by
H. Spohn. The stationary states of boundary driven interacting
particle systems have been extensively studied lately, as solvable
examples of nonequilibrium stationary states, cf. \cite{d07, bdgjl15}
and references therein.  One of the main goals is to derive in this
context the nonequilibrium functional which plays a role analogous to
the entropy in the Onsager theory of nonequilibrium thermodynamics
\cite{o31, om53}.

This program has been achieved for a class of boundary driven
interacting particle systems in two different ways. On the one hand,
relying on Derrida's formula for the stationary state as a product of
matrices, Derrida, Lebowitz and Speer \cite{dls02} obtained an
explicit formula for the nonequilibrium free energy of one-dimensional
boundary driven exclusion processes. On the other hand and by the same
time, Bertini et al. \cite{bdgjl02} derived the same formula for the
nonequilibrium free energy by solving the Hamilton-Jacobi equation for
the quasi-potential associated to the dynamical large deviations
principle for the empirical density.

Both approaches rely on the deduction of a large deviations principle
for the empirical density under the stationary state. The proof of
this result depends strongly on the non-conservative boundary dynamics
which models the interaction of the system with the reservoirs, and it
has been achieved only for dynamics which satisfy the detailed balance
conditions with respect to some Gibbs measure. If this dynamics is
slightly perturbed, Derrida's formula for the stationary state as a
product of matrices is no more available, and a bound for the entropy
production, one of the fundamental ingredients in the proof of the one
and two blocks estimates, is no more available.

We examine in this article the asymptotic behavior of the empirical
density under the stationary state of boundary driven exclusion
processes whose boundary dynamics do not satisfy a detailed balance
condition.  We consider three classes of interaction. The first one
consists of all boundary dynamics whose generator does not increase
the degree of functions of degree $1$ and $2$. The second class
includes all dynamics whose interaction with the reservoirs depends
weakly on the configuration. Finally, the third class comprises all
exclusion processes whose boundary dynamics is speeded-up. Using
duality techniques, we prove a law of large numbers for the empirical
measure under the stationary state for these three types of
interaction with the reservoirs.

Asymmetric exclusion dynamics on the positive half-line with complex
boundary dynamics have been considered by Sonigo \cite{s09}.

\section{Notation and Results}
\label{sec4}

Consider the symmetric, simple exclusion process on
$\Lambda_N=\{1,\dots ,N-1\}$ with reflecting boundary conditions. This
is the Markov process on $\Omega_N=\{0,1\}^{\Lambda_N}$ whose
generator, denoted by $L_{b,N}$, is given by
\begin{equation}
\label{02}
(L_{b,N}f)(\eta)\;=\; \sum_{k=1}^{N-2}
\{f(\sigma^{k,k+1}\eta)-f(\eta)\}\;.
\end{equation}
In this formula and below, the configurations of $\Omega_N$ are
represented by the Greek letters $\eta$, $\xi$, so that $\eta_k=1$ if
site $k\in\Lambda_N$ is occupied for the configuration $\eta$ and
$\eta_k=0$ otherwise. The symbol $\sigma^{k,k+1}\eta$ represents the
configuration obtained from $\eta$ by exchanging the occupation
variables $\eta_k$, $\eta_{k+1}$:
\begin{equation*}
(\sigma^{k,k+1}\eta)_j=
\begin{cases}
\eta_{k+1} & \mbox{ if } j=k\\
\eta_k & \mbox{ if } j=k+1\\
\eta_j & \mbox{ if } j\in \Lambda_N\setminus \{k,k+1\}\;. 
\end{cases}
\end{equation*}

This dynamics is put in contact at both ends with non-conservative
dynamics. On the right, it is coupled to a reservoir at density
$\beta\in (0,1)$. This interaction is represented by the generator
$L_{r,N}$ given by
\begin{equation}
\label{03}
(L_{r,N}f)(\eta)= \{ \beta (1-\eta_{N-1})+(1-\beta)\eta_{N-1}\}
\, \{f(\sigma^{N-1}\eta)-f(\eta)\} \;,
\end{equation}
where $\sigma^k\eta$, $k\in \Lambda_N$, is the configuration obtained
from $\eta$ by flipping the occupation variable $\eta_k$,
\begin{equation*}
(\sigma^k\eta)_j=\begin{cases}
1-\eta_k & \mbox{ if } j=k\\
\eta_j & \mbox{ if } j\in \Lambda_N\setminus \{k\}\;. 
\end{cases}
\end{equation*}

On the left, the system is coupled with different non-conservative
dynamics. The purpose of this paper is to investigate the stationary
state induced by these different interactions.

\subsection{Boundary dynamics which do not increase degrees}\label{ss1}
The first left boundary dynamics we consider are those which keep the
degree of functions of degree $1$ and $2$: those whose generator,
denoted by $L_{l,N}$, are such that for all $j\not = k$,
\begin{equation}
\label{f15a}
\begin{gathered}
L_{l,N} \, \eta_j \;=\; a^j \;+\; \sum_\ell a^j_\ell \, \eta_\ell\;,
\\
L_{l,N} \, \eta_j \, \eta_k \;=\; b^{j,k} \;+\; \sum_\ell b^{j,k}_\ell \, \eta_\ell
\;+\; \sum_{\ell,m} b^{j,k}_{\ell,m} \, \eta_\ell \, \eta_m 
\end{gathered}
\end{equation}
for some coefficients $a^j$, $a^j_\ell$, $b^{j,k}$, $b^{j,k}_\ell$,
$b^{j,k}_{\ell,m}$. 

Fix $p\ge 0$, and let $\Lambda_p^* = \{-p, \dots, 0\}$,
$\Omega^*_p = \{0,1\}^{\Lambda_p^*}$. Consider the generators of
Markov chains on $\Omega^*_p$ given by
\begin{equation*}
(L_Rf)(\eta) \; =\; \sum_{j\in \Lambda_p^*} r_j\,
\big[\, \alpha_j\, (1-\eta_j) \,+\, \eta_j\, (1-\alpha_j)\, \big] \,  
\{\, f(\sigma^{j}\eta)-f(\eta)\, \}\;,
\end{equation*}
\begin{equation*}
(L_Cf)(\eta) \; =\; \sum_{j\in \Lambda_p^*} 
\sum_{k\in \Lambda_p^*} c_{j,k}\,
\big[\,\eta_k\, (1-\eta_j) \,+\, \eta_j\, (1-\eta_k) \, \big]\,  
\{\, f(\sigma^{j}\eta)-f(\eta)\, \}\;,
\end{equation*}
\begin{equation*}
  (L_Af)(\eta) \; =\; \sum_{j \in \Lambda_p^*} \sum_{k \in \Lambda_p^*} a_{j,k}\,
  \big[\, \eta_k\, \eta_j \,+\, (1-\eta_j)\, (1-\eta_k)\, \big]\,   
  \{\, f(\sigma^{j}\eta)-f(\eta)\, \}\;.
\end{equation*}
In these formulae and below, $r_j$, $ c_{j,k}$ and $a_{j,k}$ are
non-negative constants, $0\le \alpha_j \le 1$, and $c_{j,j} =
a_{j,j}=0$ for $j\in \Lambda_p^*$.

The generator $L_R$ models the contact of the system at site $j$ with
an infinite reservoir at density $\alpha_j$. At rate $r_j\ge 0$, a
particle, resp. a hole, is placed at site $j$ with probability
$\alpha_j$, resp. $1-\alpha_j$. The generator $L_C$ models a
replication mechanism, at rate $c_{j,k}\ge 0$, site $j$ copies the
value of site $k$.  The generator $L_A$ acts in a similar way. At rate
$a_{j,k}\ge 0$, site $j$ copies the inverse value of site $k$. We add
to these dynamics a stirring evolution which exchange the occupation
variables at nearest-neighbor sites:
\begin{equation*}
(L_Sf)(\eta) \; =\; \sum_{j=-p}^{-1}
\{f(\sigma^{j,j+1}\eta)-f(\eta)\} \;.
\end{equation*}

The evolution at the left boundary we consider consists in the
superposition of the four dynamics introduced above. The generator,
denoted by $L_l$, is thus given by
\begin{equation*}
L_l \;=\; L_S \;+\; L_R \;+\; L_C \;+\; L_A \;.
\end{equation*}

Denote by $L_G$ the generator of a general Glauber dynamics on
$\Omega^*_p$: 
\begin{equation}
\label{f16}
(L_{G}f)(\eta) \; =\; \sum_{k=-p}^{0} c_k(\eta)\,  
\{f(\sigma^{k}\eta)-f(\eta)\}\;,
\end{equation}
where $c_k$ are non-negative jump rates which depend on the entire
configuration $(\eta_{-p}, \dots, \eta_0)$. We prove in Lemma
\ref{ll05} that any Markov chain on $\Omega^*_p$ whose generator
$L_{D}$ is given by $L_{D} = L_S + L_G$ and which fulfills conditions
\eqref{f15a} can be written as $L_S + L_R + L_C + L_A$ [we show that
there are non-negative parameters $r_j$, $c_{j,k}$, $a_{j,k}$ such that
$L_G = L_R + L_C + L_A$]. Therefore, by examining the Markov chain
whose left boundary condition is characterized by the generator $L_l$
we are considering the most general evolution in which a stirring
dynamics is superposed with a spin flip dynamics which fulfills
condition \eqref{f15a}.

We prove in Lemma \ref{ll04} that the Markov chain induced
by the generator $L_l$ has a unique stationary state if
\begin{equation}
\label{f15}
\sum_{j\in\Lambda^*_p} r_j \;+\; 
\sum_{j\in \Lambda_p^*} \sum_{k\in \Lambda_p^* } a_{j,k}  \;>\;0 \;.
\end{equation}
Assume that this condition is in force. Denote by $\mu$ the unique
stationary state, and let
\begin{equation}
\label{f03a}
\rho (k)\;=\; E_{\mu}[\eta_k]\;, \quad k\in\Lambda^*_p\;,
\end{equation}
be the mean density at site $k$ under the measure $\mu$. Clearly,
$0\le \rho (k) \le 1$ for all $k\in\Lambda^*_p$.  Since $E_{\mu}[L_l
\eta_j]=0$, a straightforward computation yields that
\begin{equation}
\label{f05a}
0 \;=\; r_j \, [\alpha_j - \rho(j)]  \,+\, (\ms C \rho )(j)
\;+\; (\ms A\rho)(j) \;+\; (\ms T \rho)(j) \;, \quad
j\in \Lambda^*_p\;,
\end{equation}
where 
\begin{equation*}
(\ms C \rho )(j) \;=\; \sum_{k\in \Lambda^*_p} c_{j,k}\, [\rho(k) - \rho(j)]
\;, \quad
(\ms A \rho )(j) \;=\; \sum_{k\in \Lambda^*_p} a_{j,k}\, [1 - \rho(k) - \rho(j)]\;,
\end{equation*}
\begin{equation*}
(\ms T \rho)(j) \;=\;
\begin{cases}
\rho(-p+1) - \rho(-p) & \text{ if $j=-p$}\;, \\
\rho(-1) - \rho(0) & \text{ if $j=0$}\;, \\
\rho(j+1) + \rho(j-1) - 2 \rho(j) & \text{ otherwise.}
\end{cases}
\end{equation*}
We prove in Lemma \ref{ll02} that \eqref{f05a} has a unique solution if
condition \eqref{f15} is in force.

Let $\Lambda_{N,p} =\{-p, \dots ,N-1\}$. Consider the boundary driven,
symmetric, simple exclusion process on $\Omega_{N,p} =
\{0,1\}^{\Lambda_{N,p}}$ whose generator, denoted by $L_{N}$, is given
by
\begin{equation}
 \label{eq:LNDef}
L_{N} \;=\; L_l \,+\, L_{0,1} \,+\,  L_{b,N} \,+\,  L_{r,N}\;,
\end{equation}
where $L_{0,1}$ represent a stirring dynamics between sites $0$ and
$1$: 
\begin{equation*}
(L_{0,1}f)(\eta) \; =\; f(\sigma^{0,1}\eta)-f(\eta)\;.
\end{equation*}
There is a little abuse of notation in the previous formulae because
the generators are not defined on the space $\Omega_{N,p}$ but on
smaller spaces. We believe, however, that the meaning is clear.

Due to the right boundary reservoir and the stirring dynamics, the
process is ergodic.  Denote by $\mu_N$ the unique stationary state,
and let
\begin{equation}
\label{01}
\rho_N(k)\;=\; E_{\mu_N}[\eta_k]\;, \quad k\in\Lambda_{N,p}\;,
\end{equation}
be the mean density at site $k$ under the stationary state. Of course,
$0\le \rho_N(k) \le 1$ for all $k\in\Lambda_{N,p}$, $N\ge 1$. We prove
in Lemma \ref{ll03} that under condition \eqref{f15} there exists a
finite constant $C_0$, independent of $N$, such that
\begin{equation*}
\big|\, \rho_N(k) \,-\, \rho(k)\,\big| \;\le\; C_0/N \;, \quad
\text{for all}\;\; -p\, \le\,  k \,\le\, 0\;,
\end{equation*}
where $\rho$ is the unique solution of \eqref{f05a}.

The first main result of this articles establishes a law of large
numbers for the empirical measure under the stationary state $\mu_N$.

\begin{theorem}
\label{mt0}
Assume that $\sum_j r_j>0$. Then, for any continuous function $G:[0,1]
\to \bb R$,
\begin{equation*}
\lim_{N\to \infty} E_{\mu_N} \Big[ \, \Big|
\frac 1N \sum_{k=1}^{N-1} G(k/N) \, [\eta_k -
\bar u (k/N)] \, \Big|\, \Big] =0\;,
\end{equation*}
where $\bar u$ is the unique solution of the linear equation
\begin{equation}
\label{07}
\begin{cases}
0=\Delta u  \,, \\
u(0)=\rho(0)\,, \;\; u(1) =\beta \,.
\end{cases}
\end{equation}
\end{theorem}

We refer to Section \ref{sec01} for the notation used in the next
remark.

\begin{remark}
\label{rm0}
We believe that Theorem \ref{mt0} remains in force if
$\sum_{j\in\Lambda^*_p} r_j=0$ and $\sum_{j, k\in \Lambda_p^*} a_{j,k}
>0$. This assertion is further discussed in Remark \ref{rm0bis}.
\end{remark}

\begin{remark}
\label{rm6}
The case $\sum_{j\in\Lambda^*_p} r_j + \sum_{j,k\in \Lambda_p^*}
a_{j,k} =0$ provides an example in which at the left boundary sites
behave as a voter model and acquire the value of one of their
neighbors.  One can generalize this model and consider an exclusion
process in which, at the left boundary, the first site takes the value
of the majority in a fixed interval $\{2, \dots, 2p\}$, the left
boundary generator being given by
\begin{equation*}
(L_l f)(\eta) \;=\; f(M\eta) - f(\eta)\;,
\end{equation*}
where $(M\eta)_k = \eta_k$ for $k\ge 2$, and
$(M\eta)_1 = \bs 1\{\sum_{2\le j\le 2p} \eta_j \ge p\}$. In this case it
is conceivable that the system alternates between two states, one in
which the left density is close to $1$ and one in which it is close to
$0$.
\end{remark}


The proof of Theorem \ref{mt0} is presented in Sections \ref{sec01} and
\ref{sec03}. It relies on duality computations. As the boundary
conditions do not increase the degrees of a function, the equations
obtained from the identities $E_{\mu_N}[L_N \eta_j]=0$, $E_{\mu_N}[L_N
\eta_j \eta_k]=0$ can be expressed in terms of the density and of the
correlation functions.  

\subsection{Small perturbations of flipping dynamics}
\label{sec:Model}

We examine in this subsection a model in which the rate at which the
leftmost occupation variable is flipped depends locally on the
configuration.  Consider the generator
\begin{equation}
\label{f13}
L_N \;=\; L_{l} \,+\, L_{b,N} \,+\, L_{r,N} \;,
\end{equation}
where $L_{b,N}$ and $L_{r,N}$ were defined in \eqref{02}, \eqref{03}. The
left boundary generator is given by
\begin{equation*}
(L_{l} f)(\eta) \;=\; c(\eta_1,\dots,\eta_p)\, [f(\sigma^{1}\eta)-f(\eta)] \;.
\end{equation*}
for some non-negative function $c:\{0,1\}^p\to \R_+$.
 
Let
\begin{equation}
\label{f12}
A\;=\; \min_{\xi\in\Omega_p} c(0, \xi)\;, \quad
B \;=\; \min_{\xi\in\Omega_p} c(1, \xi)
\end{equation}
be the minimal creation and annihilation rates, and denote by
\begin{equation*}
\lambda(0,\xi) \;:=\; c(0,\xi) \,-\, A
\;, \quad 
\lambda(1,\eta) \;:=\; c(1,\xi) \,-\, B
\end{equation*}
the marginal rates. We allow ourselves below a little abuse of
notation by considering $\lambda$ as a function defined on $\Omega_N$
and which depends on the first $p$ coordinates, instead of a function
defined on $\Omega_{p+1}$.  With this notation the left boundary
generator can be written as
\begin{equation*}
(L_{l} f)(\eta) \; =\;
\big[ A + (1-\eta_1)\, \lambda(\eta) \big] \; [f(T^1\eta)-f(\eta)]
\;+\; \big[B + \eta_1\, \lambda(\eta) \big]\; [f(T^0\eta)-f(\eta)]\;,
\end{equation*}
where for $a=0$, $1$,
\begin{equation*}
(T^a\eta)_k \;=\;
\begin{cases}
a & \text{if $k=1$,} \\
\eta_k & \text{otherwise.}
\end{cases}
\end{equation*}

The Markov chain with generator $L_N$ has a unique stationary state
because it is irreducible due to the stirring dynamics and the right
boundary condition.  Denote by $\mu_N$ the unique stationary state of
the generator $L_N$, and by $E_{\mu_N}$ the corresponding
expectation. Let $\rho_N(k) = E_{\mu_N}[\eta_k]$, $k\in\Lambda_N$.

\begin{theorem}
\label{mt1}
Suppose that 
\begin{equation}
\label{fc01}
(p-1)\, \sum_{\xi\in\Omega_p}
\{ \lambda(0,\xi)  + \lambda(1,\xi)  \} 
\;<\; A \,+\, B\;.
\end{equation}
Then, the limit
\begin{equation*}
\alpha:=\lim_{N\to \infty} \rho_N(1)
\end{equation*}
exists, and it does not depend on the boundary conditions at $N-1$.
Moreover, for any continuous function $G: [0,1]\to \bb R$,
\begin{equation*}
\lim_{N\to \infty} E_{\mu_N} \Big[ \, \Big|
\frac 1N \sum_{k=1}^{N-1} G(k/N) \, [\eta_k -
\bar u (k/N)] \, \Big|\, \Big] =0\;,
\end{equation*}
where $\bar u$ is the unique solution of the linear equation 
\eqref{07} with $\rho(0)=\alpha$.
\end{theorem}

\begin{remark}
\label{rm3}
There is not a simple closed formula for the left density $\alpha$.
By coupling, it is proven that the sequence $\rho_N(1)$ is Cauchy and
has therefore a limit. The density $\rho_N(1)$ can be expressed in
terms of the dual process, a stirring dynamics with creation and
annihilation at the boundary.
\end{remark}

\begin{remark}
\label{rm2}
A similar result holds for boundary driven exclusion processes in
which particles are created at sites $1\le k\le q$ with rates
depending on the configuration through the first $p$ sites, provided
the rates depend weakly [in the sense \eqref{fc01}] on the
configuration.
\end{remark}

\begin{remark}
\label{rm4}
One can weaken slightly condition \eqref{fc01}. For $\zeta\in
\{0,1\}^q$, $0\le q \le p-1$, let $A(\zeta) = \min_\xi c(\zeta,\xi)$,
where the minimum is carried over all configurations $\xi\in
\{0,1\}^{p-q}$. For $a=0$, $1$, and $\zeta\in \cup_{0\le q\le p-1}
\{0,1\}^q$, let $R(\zeta,a) = A(\zeta,a) - A(\zeta)\ge 0$ be the
marginal rate. The same proof shows that the assertion of Theorem
\ref{mt1} holds if
\begin{equation*}
\sum_{q=2}^p (q-1) \sum_{\zeta\in \{0,1\}^q} R(\zeta) \,<\, A\,+\, B\;.
\end{equation*}
\end{remark}

\begin{remark}
\label{rm7}
In \cite{e17}, Erignoux proves that the empirical measure evolves in
time as the solution of the heat equation with the corresponding
boundary conditions.
\end{remark}

The proof of Theorem \ref{mt1} is presented in Section \ref{sec02}. It
is based on a duality argument which consists in studying the process
reversed in time. We show that under the conditions of Theorem
\ref{mt1}, to determine the value of the occupation variable $\eta_1$
at time $0$, we only need to know from the past the behavior of the
process in a finite space-time window.

\subsection{Speeded-up boundary condition}

Recall the notation introduced in Subsection \ref{ss1}. Fix $p>1$ and
consider an irreducible continuous-time Markov chain on $\Omega^*_p$,
$p>0$. Denote by $L_l$ the generator of this process, and by $\mu$ the
unique stationary state. Let
\begin{equation}
\label{f03b}
\rho (k)\;=\; E_{\mu}[\eta_k]\;, \quad k\in\Lambda^*_p\;,
\end{equation}
be the mean density at site $k$ under the measure $\mu$. Clearly,
$0< \rho (k) < 1$ for all $k\in\Lambda^*_p$. The density can not be
$0$ or $1$ because every configuration has a strictly positive weight
under the stationary measure.

Fix a sequence $\ell_N\to\infty$, and consider the boundary driven,
symmetric, simple exclusion process on $\Omega_{N,p}$ whose generator,
denoted by $L_{N}$, is given by
\begin{equation*}
L_{N} \;=\; \ell_N\, L_l \,+\, L_{0,1} \,+\,  L_{b,N} \,+\,  L_{r,N}\;,
\end{equation*}
where $L_{0,1}$ represent a stirring dynamics between sites $0$ and
$1$, introduced below \eqref{eq:LNDef}. Note that the left boundary
dynamics has been speeded-up by $\ell_N$.

Due to the right boundary reservoir and the stirring dynamics, the
process is ergodic.  Denote by $\mu_N$ the unique stationary state,
and let
\begin{equation*}
\rho_N(k)\;=\; E_{\mu_N}[\eta_k]\;, \quad k\in\Lambda_{N,p}\;,
\end{equation*}
be the mean density at site $k$ under the stationary state.

\begin{theorem}
\label{mt2}
There exists a finite constant $C_0$, independent of $N$, such that
$|\rho_N(0)-\rho(0)|\le C_0/\sqrt{\ell_N}$. Moreover, for any continuous
function $G:[0,1] \to \bb R$,
\begin{equation*}
\lim_{N\to \infty} E_{\mu_N} \Big[ \, \Big|
\frac 1N \sum_{k=1}^{N-1} G(k/N) \, [\eta_k -
\bar u (k/N)] \, \Big|\, \Big] =0\;,
\end{equation*}
where $\bar u$ is the unique solution of the linear equation
\eqref{07}. 
\end{theorem}

\begin{remark}
\label{rm5}
The proof of this theorem is based on duality computations, and does
not requires one and two-blocks estimates. There is an alternative
proof relying on an estimate of the entropy production along the lines
presented in \cite[Proposition 2]{els90}, \cite[Proposition
3.3]{lov98}. This proof applies to gradient and non-gradient models
\cite{klo95}, but it requires $\ell_N$ to grow at least as $N$.
\end{remark}

The proof of Theorem \ref{mt2} is presented in Section \ref{sec5}.  As
the boundary condition has been speeded-up, each time the occupation
variables $\eta_0$, $\eta_1$ are exchanged, the distribution of the
variable $\eta_0$ is close to its stationary distribution with respect
to the left-boundary dynamics.

\section{Proof of Theorem \ref{mt0}: one point functions}
\label{sec01}

We prove in this section that the density of particles under the
stationary state $\mu_N$ is close to the solution of the linear
parabolic equation \eqref{07}.  We first show that the left boundary
dynamics we consider is indeed the most general one which does not
increase the degree of functions of degree $1$ and $2$.

For $A\subset \Lambda^*_p$, let $\Psi_A : \Omega^*_p \to \bb R$ be
given by $\Psi_A(\eta) = \prod_{k\in A} \eta_k$.  Clearly, any
function $f: \Omega^*_p \to \bb R$ can be written as a linear
combination of the functions $\Psi_A$. A function $f$ is said to be a
monomial of order $n$ if it can be written as a linear combination of
functions $\Psi_A$ where $|A|=n$ for all $A$. It is said to be a
polynomial of order $n$ if it can be written as a sum of monomials of
order $m\le n$.

Recall the definition of the generator $L_G$ given in \eqref{f16}. Fix
$-p\le k\le 0$, and write the jump rate $c_k$ as
\begin{equation*}
c_k \;=\; \sum_{A\subset \Lambda_p^*} R_{k,A} \, \Psi_A\;,
\end{equation*}
where the sum is carried over all subsets $A$ of $\Lambda_p^*$.

\begin{lemma}
\label{ll01}
The functions $L_G \Psi_{\{j\}}$, resp. $L_G \Psi_{\{j,k\}}$, $-p\le
j\not = k\le 0$, are polynomials of order $1$, resp. of order $2$, if
and only if there exists constants $R_{l,\varnothing}$, $R_{l,\{m\}}$,
$l$, $m\in \Lambda_p^*$ such that
\begin{equation}
\label{f01}
c_j(\eta) \;=\; R_{j,\varnothing} \;+\; R_{j,\{j\}}\, \eta_j  \;+\; 
\sum_{k: k\not = j} R_{j,\{k\}}\, \eta_k \, (1-2\eta_j)\;.
\end{equation}
\end{lemma}

\begin{proof}
Fix $j\in \Lambda_p^*$. A straightforward computation shows that
\begin{equation*}
L_G \Psi_{\{j\}} \;=\; \sum_{A\not\ni j}  R_{j,A} \, \Psi_A \;-\;
\sum_{A\not\ni j}  (2 R_{j,A} + R_{j,A\cup \{j\}}) \, \Psi_{A\cup \{j\}}\;.
\end{equation*}
Hence, $L_G \Psi_{\{j\}}$ is a polynomial of order $1$ if and only if
$R_{j,B} = R_{j,B\cup \{j\}} = 0$ for all $B\subset \Lambda_p^*$ such
that $|B|\ge 2$, $j\not \in B$. This proves that $L_G \Psi_{\{j\}}$ is
a polynomial of order $1$ if and only if condition \eqref{f01} holds.

If the rates are given by \eqref{f01}, for all $j\not = k \in
\Lambda_p^*$, 
\begin{equation*}
(L_G \Psi_{\{j\}})(\eta) \; =\; R_{j ,\varnothing} \,(1- 2 \eta_j)
\;-\; R_{j,\{j\}} \, \eta_j
\;+\;  \sum_{\ell: \ell \not = j} R_{j,\{\ell\}}\, \eta_\ell \;,
\end{equation*}
and 
\begin{align*}
(L_G \Psi_{\{j,k\}})(\eta) \; & =\; R_{j ,\varnothing} \,(1- 2 \eta_j)\, \eta_k
\;+\; R_{k,\varnothing} \,(1- 2 \eta_k)\, \eta_j
\;-\; \big( R_{j,\{j\}} + R_{k,\{k\}} \big) \, \eta_j \eta_k \\
\;& +\;  \sum_{\ell: \ell \not = j, k} R_{j,\{\ell\}}\, \eta_k \, \eta_\ell
\;+\; \sum_{\ell: \ell \not = j, k} R_{k,\{\ell\}}\, \eta_j \, \eta_\ell \;,
\end{align*}
which is a polynomial of degree $2$. This proves the lemma.
\end{proof}

\noindent{\bf Note:} Observe that at this point we do not make any
assertion about the sign of the constants $R_{j ,\varnothing}$,
$R_{j,\{k\}}$. \smallskip

The next result states that a generator $L_G$ whose rates satisfy
condition \eqref{f01} can be written as $L_R + L_C + L_A$.  Denote by
$\bb P_j$, resp. $\bb N_j$, $-p \le j\le 0$, the subset of points
$k\in \Lambda^*_p\setminus \{j\}$, such that $R_{j,\{k\}}\ge 0$,
resp. $R_{j,\{k\}}< 0$.

\begin{lemma}
\label{ll05}
The rates $c_j(\eta)$ given by \eqref{f01} are non-negative if and
only if
\begin{equation*}
\begin{gathered}
p_j \;:=\; R_{j,\varnothing}  \;+\; R_{j,\{j\}}\;-\;
\sum_{k\in \bb P_j} R_{j,\{k\}} \;\geq\; 0 \;, \\
q_j \;:=\; R_{j,\varnothing}\;+\;
\sum_{k\in \bb N_j} R_{j,\{k\}} \;\geq\; 0\; .
\end{gathered}
\end{equation*}
In this case, there exist non-negative rates $r_j$, $c_{j,k}$,
$a_{j,k}$ and densities $\alpha_j\in [0,1]$, $k\not = j\in
\Lambda_p^*$, such that for all $j\in \Lambda_p^*$, $\eta\in
\Omega^*_p$,
\begin{equation*}
\begin{aligned}
c_j(\eta) \; & =\; r_j\, \big[\, \alpha_j\, (1-\eta_j) \,+\,
(1-\alpha_j)\, \eta_j \,\big]
\;+\; \sum_{k\in \Lambda_p^*} c_{j,k} \, \big[\,
\eta_j\, (1-\eta_k) \,+\, \eta_k\, (1-\eta_j) \,\big]\; ,\\
\;& +\; \sum_{k\in \Lambda_p^*} a_{j,k} \, \big[\, \eta_j\, \eta_k 
\,+\, (1-\eta_k)\,(1-\eta_j) \,\big]\; .
\end{aligned}
\end{equation*} 
\end{lemma}

\begin{proof}
The first assertion of the lemma is elementary and left to the reader.
For $j\not = k\in \Lambda_p^*$, define 
\begin{gather*}
c_{j,k} \;=\; R_{j,\{k\}} \, \bs 1\{k\in \bb P_j \}
\;\geq\;  0,\quad
a_{j,k} \;=\; -\, R_{j,\{k\}} \, \bs 1\{k\in\bb N_j\}
\;\geq \;0\;, \\
r_j\, :=\, p_j \,+\, q_j \;\geq\; 0\; ,
\quad \alpha_j \;:=\; \frac{q_j}{p_j+q_j} \, \bs 1\{r_j\neq 0\} \,\in\,
[0,1] \;.
\end{gather*}
It is elementary to check that the second assertion of the lemma holds
with these definitions.
\end{proof}

\begin{lemma}
\label{ll04}
The Markov chain induced by the generator $L_l$ has a unique
stationary state if $\sum_{j\in\Lambda^*_p} r_j + \sum_{j,
  k\in\Lambda^*_p} a_{j,k} >0$. In contrast, if
$\sum_{j\in\Lambda^*_p} r_j + \sum_{j,k\in\Lambda^*_p} a_{j,k} =0$ and
$\sum_{j,k\in\Lambda^*_p} c_{j,k} >0$, then the Markov chain induced
by the generator $L_l$ has exactly two stationary states which are the
Dirac measures concentrated on the configurations with all sites
occupied or all sites empty.
\end{lemma}

\begin{proof}
Assume first that $\sum_{j\in\Lambda^*_p} r_j >0$.  Let $j\in
\Lambda_p^*$ such that $r_j>0$. If $\alpha_j> 0$, the configuration in
which all sites are occupied can be reached from any configuration by
moving with the stirring dynamics each empty site to $j$, and then
filling it up with the reservoir. This proves that under this
condition there exists a unique stationary state concentrated on
the configurations which can be attained from the configuration in which
all sites are occupied. Analogously, if $\alpha_j=0$, the
configuration in which all sites are empty can be reached from any
configuration. 

Suppose that $\sum_{j\in\Lambda^*_p} r_j =0$ and $\sum_{ j,k
  \in\Lambda^*_p} a_{j,k} >0$. We claim that from any configuration we
can reach any configuration whose total number of occupied sites is
comprised between $1$ and $|\Lambda^*_p|-1 = p$. Since the stirring
dynamics can move particles and holes around, we have only to show
that it is possible to increase, resp. decrease, the number of
particles up to $|\Lambda^*_p|-1$, resp. $1$.

Let $k\not = j\in \Lambda_p^*$ such that $a_{j,k} >0$. To increase the
number of particles up to $|\Lambda^*_p|-1$, move the two empty sites
to $j$ and $k$, and create a particle at site $j$. Similarly one can
decrease the number of particles up to $1$. This proves that under the
previous assumptions there exists a unique stationary state
concentrated on the set of configurations whose total number of
particles is comprised between $1$ and $|\Lambda^*_p|-1$.

Assume that $\sum_{j\in\Lambda^*_p} r_j =0$, $\sum_{j,
  k\in\Lambda^*_p} a_{j,k} =0$ and $\sum_{j, k\in\Lambda^*_p} c_{j,k}
>0$. In this case, the configuration with all sites occupied and the
one with all sites empty are absorbing states.  Let $k\not = j\in
\Lambda_p^*$ such that $c_{j,k} >0$. If there is at least one
particle, to increase the number of particles, move the empty site to
$j$, the occupied site to $k$, and create a particle at site
$j$. Similarly, we can decrease the number of particle if there is at
least one empty site.  This proves that in this case the set of
stationary states is a pair formed by the configurations with all
sites occupied and the one with all sites empty.
\end{proof}

\begin{lemma}
\label{ll02}
Suppose that $\sum_{j\in\Lambda^*_p} r_j + \sum_{j, k\in\Lambda^*_p}
a_{j,k}>0$. Then, there exists a unique solution to \eqref{f05a}.
\end{lemma}

\begin{proof}
Equation \eqref{f03a} provides a solution and guarantees existence. We
turn to uniqueness.  Suppose first that $\sum_{j\in\Lambda^*_p} r_j>0$
and $\sum_{j,k\in\Lambda^*_p} a_{j,k} =0$. In this case, the
operator $\ms A$ vanishes.  Consider two solution $\rho^{(1)}$,
$\rho^{(2)}$, and denote their difference by $\gamma$. The difference
satisfies the linear equation
\begin{equation*}
0 \;=\; -\, r_j \, \gamma(j) \;+\; (\ms C \gamma)(j) \;+\; (\ms T
\gamma)(j)\;, \quad j\in \Lambda^*_p\;.
\end{equation*}
Let $\pi$ be the unique stationary state of the random walk on
$\Lambda^*_p$ whose generator is $\ms C + \ms T$. Multiply both sides
of the equation by $\gamma(j)\,\pi(j)$ and sum over $j$ to obtain that
\begin{align*}
0 \;=\; -\, \sum_{j\in \Lambda^*_p} r_j \, \gamma(j)^2 \pi(j) 
\;+\; \< (\ms C + \ms T) \gamma \,,\,\gamma\> \;,
\end{align*}
where $\< f,g\>$ represents the scalar product in $L^2(\pi)$.  As all
terms on the right-hand side are negative, the identity
$\< (\ms C + \ms T) \gamma \,,\,\gamma\> =0$ yields that $\gamma$ is
constant. Since, by hypothesis, $\sum_j r_j >0$, $\gamma\equiv 0$,
which proves the lemma.

Suppose next that $\sum_{j\in\Lambda^*_p} r_j> 0$ and
$\sum_{j,k\in\Lambda^*_p} a_{j,k} >0$. Define the rates $t_{j,k}\ge
0$, $j\not = k \in\Lambda^*_p$, so that
\begin{equation*}
(\ms T\, f) (j) \;=\; \sum_{k: k \not =j} t_{j,k}\,
[f(k)- f(j)] \;, \quad j\in \Lambda^*_p \;.
\end{equation*}

Let $\Lambda^{\rm ext}_p = \{-1,1\} \times \Lambda^*_p$. Points in
$\Lambda^{\rm ext}_p$ are represented by the symbol $(\sigma, k)$,
$\sigma = \pm 1$, $-p\le k\le 0$. We extend the definition of a
function $f: \Lambda^*_p \to \bb R$ to $\Lambda^{\rm ext}_p$ by
setting $f(1,k) = f(k)$, $f(-1,k) = 1 - f(k)$, $k\in
\Lambda^*_p$. This new function is represented by $\widehat f :
\Lambda^{\rm ext}_p \to \bb R$.

With this notation we may rewrite equation \eqref{f05a} as
\begin{equation}
\label{f05b}
0 \;=\; r_{(1,j)} \, [\alpha_{(1,j)} - \widehat \rho(1,j)]  \,+\, 
(\widehat{\ms C} \, \widehat \rho )\, (1,j) \;+\; (\widehat{\ms A} \, 
\widehat \rho)\, (1,j)  \;+\; (\widehat{\ms T} \, \widehat \rho)\, (1,j) \,, 
\,\; j\in \Lambda^*_p\;,
\end{equation}
where, $r_{(1,j)} = r_j$, $\alpha_{(1,j)} = \alpha_{j}$, 
\begin{equation*}
(\widehat{\ms A}\, \widehat \rho)\, (1,j) \;=\; \sum_{k \in \Lambda^*_p} a_{j,k}\,
[\, \widehat \rho(-1,k)-\widehat  \rho(1,j) \, ]\;, 
\end{equation*}
and $\widehat{\ms C}$, $\widehat{\ms T}$ are the generators of the
Markov chains on $\Lambda^{\rm ext}_p$ characterized by the rates
$\widehat c$, $\widehat t$ given by
\begin{align*}
& \widehat c\; [\, (\pm 1,j) \,,\, (\pm 1,k) \,] \,=\, c_{j,k}\;, \quad
\widehat c\, [\, (\pm 1,j) \,,\, (\mp 1,k) \,] \,=\, 0\;, \\
&\quad \widehat t\; [\, (\pm 1,j) \,,\, (\pm 1,k) \,] \,=\, t_{j,k}\;, \quad
\widehat t\, [\, (\pm 1,j) \,,\, (\mp 1,k) \,] \,=\, 0\;.
\end{align*}

Multiply equation \eqref{f05a} by $-1$ to rewrite it as 
\begin{equation}
\label{f05c}
0 \;=\; r_{(-1,j)} \, [ \alpha_{(-1,j)} - \widehat \rho(-1,j)]  
\,+\, (\widehat{\ms C} \, \widehat \rho )(-1,j)
\;+\; (\widehat{\ms A}\, \widehat\rho)(-1,j) \;+\; 
(\widehat{\ms T} \, \widehat \rho)(-1,j) 
\end{equation}
for any $j\in \Lambda^*_p$, where $r_{(-1,j)} = r_{j}$,
$\alpha_{(-1,j)} = 1-\alpha_j$, and
\begin{equation*}
(\widehat{\ms A}\, \widehat \rho)(-1,j) \;=\; \sum_{k \in\Lambda^*_p} a_{j,k}\,
[\, \widehat \rho(1,k)-\widehat  \rho(-1,j)\, ]\;.
\end{equation*}

Since the operator $\widehat {\ms C} + \widehat {\ms A} + \widehat
{\ms T}$ defines an irreducible random walk on $\Lambda^{\rm ext}_p$,
we may proceed as in the first part of the proof to conclude that
there exists a unique solution of \eqref{f05a}.

Finally, suppose that $\sum_{j\in\Lambda^*_p} r_j = 0$ and
$\sum_{j,k\in\Lambda^*_p} a_{j,k} >0$. Let $\rho$ be a solution to
\eqref{f05a}. Then, its extension $\widehat\rho$ is a solution to
\eqref{f05b}, \eqref{f05c}. The argument presented in the first part
of the proof yields that any solution of these equations is
constant. Since $\widehat \rho(1,k) = \rho(k) = 1 - \widehat
\rho(-1,k)$, we conclude that this constant must be $1/2$. This proves
that in the case where $\sum_{j\in\Lambda^*_p} r_j= 0$,
$\sum_{j,k\in\Lambda^*_p} a_{j,k} >0$, the unique solution to
\eqref{f05a} is constant equal to $1/2$.
\end{proof}

Recall from \eqref{01} the definition of $\rho_N$.

\begin{lemma}
\label{ll03}
Suppose that $\sum_{j\in\Lambda^*_p} r_j + \sum_{j,k\in\Lambda^*_p}
a_{j,k}>0$.  Then, for $0\le k<N$,
\begin{equation}
\label{f07}
\rho_N(k) \;=\; \frac {k}{N} \, \beta \;+\; \frac {N-k}{N} \,
\rho_N(0) \;.
\end{equation}
Moreover, there exists a finite constant $C_0$, independent of $N$, such
that 
\begin{equation*}
\big|\, \rho_N(k) \,-\, \rho(k)\,\big| \;\le\; C_0/N \;, \quad
-p\, \le\,  k \,\le\, 0\;,
\end{equation*}
where $\rho$ is the unique solution of \eqref{f05a}. 
\end{lemma}

\begin{proof}
Fix $1\le k<N$. As $\mu_N$ is the stationary state,
$E_{\mu_N}[L_N \, \eta_k]=0$. Hence, if we set $\rho_N(N)=\beta$,
$(\Delta_N \rho_N)(k) := \rho_N(k-1) + \rho_N(k+1) - 2 \rho_N(k)
=0$. In particular, $\rho_N$ solves the discrete difference equation
\begin{equation*}
(\Delta_N \rho_N)(k) \,=\, 0\;, \;\;  1 \, \le\,  k \,<\, N \;,  \quad
\rho_N(N) \,=\, \beta\;, \quad \rho_N(0) \,=\, \rho_N(0)\;,  
\end{equation*}
whose unique solution is given by \eqref{f07}. This proves the first
assertion of the lemma.

We turn to the second statement.  It is clear that $\rho_N(j)$
fulfills \eqref{f05a} for $-p\le j<0$. For $j=0$ the equation is
different due to the stirring dynamics between $0$ and $1$ induced by
the generator $L_{0,1}$. We have that
\begin{equation*}
0 \;=\; r_0 \, [\alpha_0 - \rho_N(0)]  \,+\, (\ms C \rho_N )(0)
\;+\; (\ms A\rho_N)(0) \;+\;  (\Delta_N \rho_N)(0)\;.
\end{equation*}
By \eqref{f07}, we may replace $\rho_N(1)$ by $[1-(1/N)]\, \rho_N(0) +
(1/N)\beta$, and the previous equation becomes
\begin{equation}
\label{f09}
0 \;=\; r_0 \, [\alpha_0 - \rho_N(0)]  \,+\, (\ms C \rho_N )(0)
\;+\; (\ms A\rho_N)(0) \;+\;  (\ms T \rho_N)(0)
\;+\; \frac 1N \,\big[\, \beta - \rho_N(0)\,\big ]  \;.
\end{equation}
This equation corresponds to \eqref{f05a} with $r'_0 = r_0 + (1/N)$
and $\alpha'_0 =(\alpha_0\, r_0 + \beta/N) /[r_0 +
(1/N)]$.

By Lemma \ref{ll02}, equation \eqref{f05a} for $j\not =0$ and
\eqref{f09} for $j=0$ has a unique solution. Let $\gamma_N = \rho_N -
\rho$, where $\rho$ is the solution of \eqref{f05a}. $\gamma_N$
satisfies
\begin{equation*}
0 \;=\; \frac 1N \, [\beta - \rho_N(0)] \,\delta_{0,j}
\; -\, r_j \,  \gamma_N(j) \,+\, (\ms C \gamma_N )(j)
\;+\; (\ms A \gamma_N)(j) \;+\;  (\ms T \gamma_N)(j)\;,
\end{equation*}
where $\delta_{0,j}$ is equal to $1$ if $j=0$ and is equal to $0$
otherwise. 

We complete the proof in the case $\ms A=0$. The other cases can be
handled by increasing the space, as in the proof of Lemma
\ref{ll02}. Denote by $\pi$ the stationary state of the generator $\ms
C + \ms T$. Multiply both sides of the previous equation by $\pi(j)
\gamma_N(j)$ and sum over $j$ to obtain that
\begin{equation*}
\sum_{j\in\Lambda^*_p} r_j \,  \gamma_N(j)^2 \pi(j) \,+\, 
\< -\, (\ms C + \ms T) \, \gamma_N \,,\, \gamma_N \>
\;=\; \theta_N\, \gamma_N(0) \, \pi(0)  \;,
\end{equation*}
where $\theta_N = (1/N) \, [\beta - \rho_N(0)]$. Let
$k\in \Lambda^*_p$ such that $r_k>0$. Such $k$ exists by assumption.
Rewrite $\gamma_N(0)$ as
$\sum_{k<j\le 0} [\gamma_N(j) - \gamma_N(j-1)] + \gamma_N(k)$ and use
Young's inequality to obtain that there exists a finite constant
$C_0$, depending only on $p$, $\pi$ and on the rates $c_{j,k}$, $r_j$
such that
\begin{equation*}
\theta_N\, \gamma_N(0) \, \pi(0) \;\le\;
(1/2) \, r_k \,  \gamma_N(k)^2\, \pi(k) \,+\, 
(1/2)\, \< -\, (\ms C + \ms T) \, \gamma_N \,,\, \gamma_N \>
\;+\; C_0 \, \theta_N^2\;.
\end{equation*}
Here and throughout the article, the value of the constant $C_0$ may
change from line to line. The two previous displayed equations and the
fact that $|\beta - \rho_N(0)|\le 1$ yield that
\begin{equation*}
\sum_{j\in\Lambda^*_p} r_j \,  \gamma_N(j)^2 \pi(j)
\,+\, 
\< -\, (\ms C + \ms T) \, \gamma_N \,,\, \gamma_N \>
\;\le\; \frac {C_0}{N^2} \;\cdot
\end{equation*}
In particular, $\gamma_N(k)^2 \le C_0/N^2$ and $[\gamma_N(j+1) -
\gamma_N(j)]^2 \le C_0/N^2$ for $-p \le j <0$. This completes the
proof of the lemma.
\end{proof}

\section{Proof of Theorem \ref{mt0}: two point functions.}
\label{sec03}

We examine in this section the two-point correlation function under
the stationary state $\mu_N$. Denote by $\bb D_N$ the discrete simplex
defined by
\begin{equation*}
\bb D_N \;=\; \{(j,k) : -p \le j<k\le N-1\}\quad
\text{and set}\quad 
\Xi_N = \{-1,1\} \times \bb D_N\;.
\end{equation*}
Let
\begin{equation*}
\bar\eta_m \;=\;  1 - \eta_m \;, \quad 
\bar\rho_N(m) \;=\; 1 - \rho_N(m)\;, \quad m\in \Lambda_{N,p}\;,
\end{equation*}
and define the two-point correlation function $\varphi_N(\sigma,j,k)$,
$(\sigma, j,k) \in \Xi_N$, by
\begin{equation}
\label{g10}
\begin{aligned}
& \varphi_N(1,j,k) \;=\; E_{\mu_N} \big[\, \{\eta_j-\rho_N(j)\}\,
\{\eta_k-\rho_N(k)\}\,\big] \;, \\
&\quad \varphi_N(-1,j,k) \;=\; 
E_{\mu_N} \big[\, \{\bar \eta_j- \bar \rho_N(j)\}\,
\{\eta_k-\rho_N(k)\}\,\big] \;. 
\end{aligned}
\end{equation}
Note that $\varphi_N(-1,j,k) = -\, \varphi_N(1,j,k)$.  The identity
$E_{\mu_N}[L_N \{\eta_j-\rho_N(j)\}\, \{\eta_k-\rho_N(k)\} ]=0$ provides
a set of equations for $\varphi_N$. Their exact form requires some
notation.

Denote by $\ms L^{\rm rw}_N$ the generator of the symmetric,
nearest-neighbor random walk on $\bb D_N$. This generator is defined
by the next two sets of equations. If $k-j>1$,
\begin{equation*}
(\ms L^{\rm rw}_N \phi)(j,k) \;=\; 
\begin{cases}
(\bs \Delta \phi) (j,k) & \text{if $j>-p$, $k < N-1$}, \\
(\bs \nabla^+_1 \phi)(-p,k) + (\bs \Delta_2 \phi) (-p,k) & \text{if $j=-p$, $k< N-1$}, \\
(\bs \Delta_1 \phi) (j,N-1)+ (\bs \nabla^-_2 \phi) (j,N-1)  & \text{if $j>-p$, $k=N-1$}, \\
(\bs \nabla^+_1 \phi)(-p,N-1) + (\bs \nabla^-_2 \phi) (-p,N-1) & \text{if $j=-p$, $k=N-1$}, \\
\end{cases}
\end{equation*}
while for $-p< k < N-2$,
\begin{gather*}
(\ms L^{\rm rw}_N \phi)(k,k+1) \;=\; (\bs \nabla^-_1 \phi)(k,k+1) 
\;+\;  (\bs \nabla^+_2 \phi)(k,k+1) \;, \\
(\ms L^{\rm rw}_N \phi)(-p,-p+1) \;=\;  (\bs \nabla^+_2 \phi)(-p,-p+1) \;, \\
(\ms L^{\rm rw}_N \phi)(N-2,N-1) \;=\;  (\bs \nabla^-_1 \phi)(N-2,N-1) \;.
\end{gather*}
In these formulae, $\bs \nabla^\pm_i$, resp. $\bs \Delta_i$, represents the
discrete gradients, resp. Laplacians, given by
\begin{align*}
& (\bs \nabla^\pm_1 \phi)(j,k) \;=\; \phi(j\pm 1,k) - \phi(j,k)\;, \quad
(\bs \nabla^\pm_2 \phi)(j,k) \;=\; \phi(j,k\pm 1) - \phi(j,k)\;, \\
&\quad (\bs \Delta_1 \phi) (j,k) = \phi(j-1,k) + \phi(j+1,k) - 2 \phi(j,k)\; , \\
&\qquad (\bs \Delta_2 \phi) (j,k) = \phi(j,k-1) + \phi(j,k+1)-2 \phi(j,k)\;, \\
&\qquad\quad (\bs \Delta \phi) (j,k) = (\bs \Delta_1 \phi) (j,k) + (\bs \Delta_2 \phi) (j,k)\;.
\end{align*}

Let $L^{\rm ex}_N$ be the generator given by $L^{\rm ex}_N = L_S +
L_{0,1} + L_{b,N}$.  A straightforward computation yields that for
$(j,k) \in \bb D_N$,
\begin{equation*}
E_{\mu_N} \big[\, L^{\rm ex}_N  \, \{\eta_j-\rho_N(j)\}\,
\{\eta_k-\rho_N(k)\} \, \big] 
\;=\;  (\ms L^{\rm rw}_N  \varphi_N) (1,j,k) \;+\; F_N(1,j,k)\;,
\end{equation*}
where it is understood that the generator $\ms L^{\rm rw}_N$ acts on
the last two coordinates keeping the first one fixed, and 
\begin{equation}
\label{g09}
F_N(\sigma, j,k) \;=\; -\, \sigma\, [\rho_N(j+1)-\rho_N(j)]^2 \, 
\bs 1\{k=j+1\} \;.
\end{equation}
Similarly, 
\begin{equation*}
E_{\mu_N} \big[\, L^{\rm ex}_N  \, \{\bar \eta_j- \bar \rho_N(j)\}\,
\{\eta_k-\rho_N(k)\} \, \big] 
\;=\;  (\ms L^{\rm rw}_N  \varphi_N) (- 1,j,k) \;+\; F_N(- 1,j,k)\;.
\end{equation*}
For the next generators, we do not repeat the computation of the
action of the generator on the product $\{\bar \eta_j- \bar
\rho_N(j)\}\, \{\eta_k-\rho_N(k)\}$ because it can be inferred from
the action on $\{\eta_j- \rho_N(j)\}\, \{\eta_k-\rho_N(k)\}$. \smallskip

We turn to the remaining generators. Extend the definition of the
rates $r_j$, $c_{j,k}$ and $a_{j,k}$ to $\Lambda_{N,p}$ by setting
\begin{equation*}
r_j \,=\, c_{j,k} \,=\, a_{j,k} \,=\, 0 \quad 
\text{ if $j\not\in\Lambda_p^*$  or $k\not\in\Lambda_p^*$}\; .
\end{equation*}
To present simple expressions for the equations satisfied by the
two-point correlation function, we add cemetery points to the state
space $\Xi_N$. Let $\overline{\Xi}_N = \Xi_N \cup \partial\, \Xi_N$,
where
\begin{align}
\nonumber
\partial\, \Xi_N \; & =\; \big\{ (\sigma, k) :
\sigma = \pm 1 \,,\, -p\le k<N \big\} \,\cup\,
\big\{ (\sigma, k,k) :
\sigma = \pm 1 \,,\, -p\le k\le 0 \big\} \\
\, & \cup\,
\big\{ (\sigma, k,N) :
\sigma = \pm 1 \,,\, -p\le k<N-1 \big\}
\label{eqf1}
\end{align}
 is the set of absorbing points.

A straightforward computation yields that for $(j,k) \in \bb D_N$,
\begin{equation*}
E_{\mu_N} \big[\, L_R \, \{\eta_j-\rho_N(j)\}\,
\{\eta_k-\rho_N(k)\} \, \big] 
\;=\; (\ms L^{\dagger}_R  \, \varphi_N) (1,j,k) \;,
\end{equation*}
where
\begin{equation*}
(\ms L^{\dagger}_R  \, \phi) (\sigma, j,k)  \;=\;
r_j \, [ \varphi_N(\sigma,k) - \varphi_N(\sigma,j,k) ]
\,+\, r_k \, [ \varphi_N(\sigma,j) - \varphi_N(\sigma,j,k) ]
\end{equation*}
provided we set
\begin{equation}
\label{g08}
\varphi_N(\sigma,m) \;=\; 
b_N(\sigma,m) \; :=\;  0 \;, \quad -p\,\le\, m\,<\, N\;, \;\;
\sigma \;=\; \pm 1 \;.
\end{equation}

Similarly, an elementary computation yields that for $(j,k) \in \bb
D_N$,
\begin{equation*}
E_{\mu_N} \big[\, L_{r,N} \, \{\eta_j-\rho_N(j)\}\,
\{\eta_k-\rho_N(k)\} \, \big] 
\;=\; (\ms L^{\dagger}_{r,N}  \, \varphi_N) (1,j,k) \;,
\end{equation*}
where
\begin{equation*}
(\ms L^{\dagger}_{r,N}  \, \varphi_N) (\sigma,j,k)  
\;=\; \bs 1\{k=N-1\} \, [ \varphi_N(\sigma,j,N) -
\varphi_N(\sigma,j,k)] \;, 
\end{equation*}
provided we set
\begin{equation}
\label{g07}
\varphi_N(\sigma,m,N) \;=\; b_N(\sigma,m,N) \;:=\; 0 \;, 
\quad -p\,\le\, m\,\le\, N-2\;, \;\; \sigma \;=\; \pm 1 \;.
\end{equation}

We turn to the generator $L_C$. An elementary computation yields that
for $(j,k) \in \bb D_N$,
\begin{equation*}
E_{\mu_N} \big[\, L_C \, \{\eta_j-\rho_N(j)\}\,
\{\eta_k-\rho_N(k)\} \, \big] 
\;=\;  (\ms L^{\dagger}_C  \, \varphi_N) (1,j,k) \;, 
\end{equation*}
where 
\begin{align*}
(\ms L^\dagger_C \phi)(\sigma, j,k) \; =\; 
\sum_{m: m\not = j} c_{j,m} 
\{ \phi(\sigma,m,k) - \phi(\sigma,j,k) \}
\; +\;  \sum_{m: m\not = k} c_{k,m} 
\{ \phi(\sigma,j,m) - \phi(\sigma,j,k) \}\;,
\end{align*}
provided we set
\begin{equation}
\label{g06}
\varphi_N(\sigma,m,m) \;=\; b_N(\sigma,m,m) \; :=\;
\sigma\, \rho_N(m)\, [1-\rho_N(m)]\;, \quad
-p\le m\le 0\;.
\end{equation}

Finally, we claim that 
for $(j,k) \in \bb D_N$,
\begin{equation*}
E_{\mu_N} \big[\, L_A \, \{\eta_j-\rho_N(j)\}\,
\{\eta_k-\rho_N(k)\} \, \big] 
\;=\;  (\ms L^{\dagger}_A  \, \varphi_N) (1,j,k) \;, 
\end{equation*}
where
\begin{equation*}
(\ms L^\dagger_A \phi)(\sigma,j,k) \; =\; 
\sum_{m: m\not = j} a_{j,m} \{ \phi(-\sigma,m,k) - \phi(\sigma,j,k) \}
\;+\;  \sum_{m: m\not = k} a_{k,m} \{ \phi(-\sigma,j,m) - \phi(\sigma,j,k) \}\;,
\end{equation*}
and $\varphi_N (\sigma,k,k)$ is given by \eqref{g06}.  Hence, the
generator $\ms L^\dagger_A$ acts exactly as $\ms L^\dagger_C$, but it
flips the value of the first coordinate.  Note that it is the only
generator which changes the value of the first coordinate. \smallskip

Let $\ms L^\dagger_N$ be the generator on $\overline{\Xi}_N$ given by
\begin{equation*}
\ms L^\dagger_N \;=\; \ms L^{\rm rw}_N \,+\, \ms L^\dagger_{R} 
\,+\, \ms L^\dagger_{r,N} \,+\, \ms L^\dagger_{C}
\,+\, \ms L^\dagger_{A}\;.
\end{equation*}
If $\sum_j \sum_{j, k} a_{j,k}=0$, the generator $\ms L^\dagger_{A}$
vanishes, the first coordinate is kept constant by the dynamics and we
do not need to introduce the variable $\sigma$.  Note that the points
in $\partial\, \Xi_N$ are absorbing points.

As $E_{\mu_N}[L_N \{\eta_j-\rho_N(j)\}\, \{\eta_k-\rho_N(k)\} ]=0$,
the previous computations yield that the two-point correlation
function $\varphi_N$ introduced in \eqref{g10} solves
\begin{equation}
\label{g11}
\begin{cases}
\vphantom{\Big\{}
(\ms L^\dagger_N \psi_N) (\sigma,j,k) + F_N (\sigma,j,k) = 0 \;,
\;\; (\sigma,j,k) \in \Xi_N\;, \\
\vphantom{\Big\{}
\psi_N(\sigma,j,k) \,=\,  b_N(\sigma,j,k) \;,
\;\; (\sigma,j,k) \in \partial \, \Xi_N\;,
\end{cases}   
\end{equation}
where $F_N$ and $b_N$ are the functions defined in \eqref{g09},
\eqref{g08}, \eqref{g07}, \eqref{g06}.

As $\ms L^\dagger_N$ is a generator, \eqref{g11} admits a unique
solution [on the set $\{(1,j,k) : (j,k)\in\bb D_N\}$ if $\ms
L^\dagger_{A}$ vanishes]. This solution can be represented in terms of
the Markov chain induced by the generator $\ms L^\dagger_N$. 

Denote by $\varphi^{(1)}_N$, resp. $\varphi^{(2)}_N$, the solution of
\eqref{g11} with $b_N=0$, resp. $F_N=0$. It is clear that $\varphi_N
= \varphi^{(1)}_N + \varphi^{(2)}_N$. Denote by $X_N(t)$ the
continuous-time Markov chain on $\overline{\Xi}_N$ associated to the
generator $\ms L^\dagger_N$. Let $\bs P_{(\sigma,j,k)}$ be the
distribution of the chain $X_N$ starting from
$(\sigma,j,k)$. Expectation with respect to $\bs P_{(\sigma,j,k)}$ is
represented by $\bs E_{(\sigma,j,k)}$.

Let $H_N$ be the hitting time of the boundary $\partial\, \Xi_N$:
\begin{equation*}
H_N\;=\; \inf \big\{t\geq 0 : X_N(t) \in \partial\, \Xi_N \,\big\}\;.
\end{equation*}
It is well known (cf. \cite[Theorem 6.5.1]{fri75} in the continuous
case) that
\begin{equation*}
\varphi^{(1)}_N (\sigma,j,k)  \;=\;
\bs E_{(\sigma,j,k)} \Big[ \int_0^{H_N} F_N(X_N(s)) \, ds \Big]\;.
\end{equation*}
 It is also well known that
\begin{equation*}
\varphi^{(2)}_N (\sigma,j,k)  \;=\;
\bs E_{(\sigma,j,k)} \big[ b_N(X_N(H_N)) \big]\;.
\end{equation*}

To estimate $\varphi^{(1)}_N$ and $\varphi^{(2)}_N$ we need to show
that the process $X_N(t)$ attains the boundary $\partial\, \Xi_N$ at
the set $\{ (\sigma, k,k) : \sigma = \pm 1 \,,\, -p\le k\le 0 \} $
with small probability. This is the content of the next two lemmata.

For a subset $A$ of $\overline{\Xi}_N$, denote by $H(A)$,
resp. $H^+(A)$, the hitting time of the set $A$, resp. the return time
to the set $A$:
\begin{equation*}
H(A) \;=\; \inf \big\{t\ge 0 : X_N(t) \in A
\big\} \;, \quad 
H^+(A) \;=\; \inf \big\{t\ge \tau_1 : X_N(t) \in A
\big\}\;,
\end{equation*}
where $\tau_1$ represents the time of the first jump: $\tau_1 = \inf
\{s>0 : X_N(s) \not = X_N(0)\}$. 


The next lemma, illustrated in Figure \ref{fig:lemmal07}, translates
to the present model the fact that starting from $(1,0)$ the
two-dimensional, nearest-neighbor, symmetric random walk hits the line
$\{(0,k) : k\in \bb Z\}$ at a distance $n$ or more from the origin
with a probability less than $C/n$.

Let $\widehat{\bs Q}_{(l,m)}$ be the law of such a random walk
evolving on $\bb Z^2$ starting from $(l,m)$.  Denote by $B_r(l,m)$ the
ball of radius $r>0$ and center $(l,m)\in \bb Z^2$, and by $\bb L$ the
segment $\{(\sigma, 0,a) : \sigma = \pm 1 \,,\, 1\le a<N \}$.
Represent the coordinates of $X_N(t)$ by $(\sigma_N(t), X^1_N(t),
X^2_N(t))$.

\begin{figure}
\includegraphics{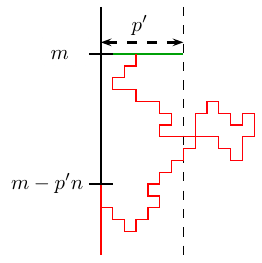}
\caption{Lemma \ref{l07} states that a random walk (red
    trajectory) started from the green segment has a probability at
    most of order $1/n$ of hitting $\bb L$ in the red half-line.}
\label{fig:lemmal07} 
\end{figure}

\begin{lemma}
\label{l07}
Let $p'=p+1$. There exists a finite constant $C_0$ such that for all $n$,
\begin{equation*}
\max_{\sigma = \pm 1} \max_{l,m} 
\bs P_{(\sigma,l,m )} \big[\, H(\bb L)  = \infty
\;\text{or}\; X^2_N(H(\bb L)) \le m - p'n \,\big]
\;\le\; \frac {C_0}n\;,
\end{equation*}
where the maximum is carried over all pairs $(l,m)$ such that
$1\le l\le p'$, $\{(a,b) \in B_{p'n}(0,m) : a\ge 0\} \subset \bb D^0_N
= \{(a,b) \in \bb D_N : a\ge 0\}$.
\end{lemma}

\begin{proof}
Let $\bb L_r = \{(0,l) : -r\le l\le r\}$.  By \cite[Proposition
2.4.5]{law91}, there exists a finite constant $C_0$ such that for all
$n\ge 1$,
\begin{equation*}
\widehat{\bs Q}_{(1,0)} \big[ H(B_n(0,0)^\complement) < H (\bb L_{n}) \,\big]
\;\le\; \frac {C_0}n\;\cdot 
\end{equation*}

Let $\bb L_r (l,m) = \{(\sigma, l,a) : \sigma = \pm 1 \,,\, m-r\le
a\le m+r\}$. By the previous displayed equation, if $\bb L_{n} (l,m)$
is contained in $\bb D^0_N$,
\begin{equation*}
\bs P_{(\sigma,l+1,m )} \big[\, H(B_{n}(l,m)^\complement) 
< H (\bb L_{n} (l,m)) \,\big] \;\le\; \frac {C_0}n\;\cdot
\end{equation*}
Iterating this estimate $i$ times yields that
\begin{equation*}
\bs P_{(\sigma,l+i,m )} \big[\, H(B_{in}(l,m)^\complement) 
< H (\bb L_{in} (l,m)) \,\big]
\;\le\; \frac {C_0\, i}n
\end{equation*}
provided all sets appearing in this formula are contained in $\bb
D^0_N$. The assertion of the lemma follows from this estimate
 and the following observation: 
\begin{equation*}
\{\, H(\bb L)  = \infty
\;\text{or}\; X^2_N(H(\bb L)) \le m - p'n \,\}\subseteq
\{\, H(B_{p'n}(0,m)^\complement) 
< H (\bb L_{p'n} (0,m)) \,\}.
\end{equation*}
\end{proof}

The next lemma presents the main estimate needed in the proof of the
bounds of the two-point correlation functions. Recall from
\eqref{eqf1} that we denote by $(\sigma, k)$, $(\sigma, k, N)$ some
cemitery points. Let
\begin{gather*}
\Sigma = \{ (\sigma, l,0) : \sigma = \pm 1 \,,\, -p\le l<0\}\;, \\
\partial_N \,=\, \big\{ (\sigma, k) :
\sigma = \pm 1 \,,\, -p\le k<N \big\} 
\, \cup\, \big\{ (\sigma, k,N) :
\sigma = \pm 1 \,,\, -p\le k<N-1 \big\}\;.
\end{gather*}

\begin{lemma}
\label{l08}
For all $\delta>0$, 
\begin{equation*}
\lim_{N\to\infty} \max_{\substack{(j,k) \in \bb D_N \\ j>\delta N}}
\bs P_{(1,j,k)} \big[ H(\Sigma) < H(\partial_N) \big] \;=\; 0\;.
\end{equation*}
\end{lemma}

\begin{proof}
Fix $\delta>0$ and $(j,k) \in \bb D_N$ such that $j>\delta N$. 
Let 
\begin{equation*}
\partial^0_N \;=\; \{ (\sigma, 0, m) : \sigma = \pm 1 \,,\, 0<m<N\} 
\,\cup\, \big\{ (\sigma, k,N) :
\sigma = \pm 1 \,,\, -p\le k<N-1 \big\} \;,
\end{equation*}
and set $\tau = H(\partial^0_N)$. Clearly, $\tau<H(\Sigma)$. Hence, by
the strong Markov property, the probability appearing in the statement
of the lemma is equal to
\begin{equation}
\label{g17b}
\bs E_{(1,j,k)} \Big[ \bs P_{X_N(\tau)} \big[ H(\Sigma) < H(\partial_N) \big]\,
\Big]\;. 
\end{equation}
Up to time $\tau$, the process $X_N$ evolves as a symmetric random
walk on $\bb D_N$

Let $\ell_N$ be a sequence such that $\ell_N \ll N$. We claim that for
all $\delta>0$,
\begin{equation}
\label{g16b}
\lim_{N\to\infty} \max_{(l,m)} 
\, \bs P_{(1,l,m)} \big[ \, X^2_N(\tau) \le \ell_N \,
\big] \;=\;0\;,
\end{equation}
where the maximum is carried out over all pairs $(l,m)\in \bb D_N$
such that $l>\delta N$.  The proof of this statement relies on the
explicit form of the harmonic function for a $2$-dimensional Brownian
motion.

Up to time $\tau$, the process $Y_N(t) = (X^1_N(t), X^2_N(t))$ evolves
on the set $\triangle_N = \{(a,b) : 0\le a<b\le N\}$. Let $\Box_N =
\{0, \dots, N-1\}\times \{1, \dots, N\}$. Denote by $Z_N(t) =
(Z^1_N(t), Z^2_N(t))$ the random walk on $\Box_N$ which jumps from a
point to any of its neighbors at rate $1$. Let $\Phi_N: \Box_N \to
\triangle_N$ the projection defined by $\Phi_N(a,b)=(a,b)$ if
$(a,b)\in \triangle_N$, and $\Phi_N(a,b)=(b-1,a+1)$ otherwise. The
process $\Phi_N(Z_N(t))$ does not evolve as $Y_N(t)$ because the jumps
of $\Phi_N(Z_N(t))$ on the diagonal $\{(d,d+1) : 0\le d<N\}$ are
speeded-up by $2$, but the sequence of sites visited by both processes
has the same law. Therefore,
\begin{equation*}
\bs P_{(1,l,m)} \big[ \, X^2_N(\tau) \le \ell_N \,
\big] \;=\; \bs Q_{(l,m)} \big[ \, Z_N(\widehat \tau)
\in \, \angle_N \big]\;,
\end{equation*}
where $\bs Q_{(l,m)}$ represents the law of the process $Z_N$ starting
from $(l,m)$, $\widehat \tau$ the hitting time of the boundary of
$\Box_N$ and $\angle_N$ the set $\{(0,a) : 1\le a\le \ell_N\} \cup
\{(b,1) : 0\le b\le \ell_N-1\}$.

Denote by $B(r) \subset \bb R^2$, $r>0$, the ball of radius $r$
centered at the origin. In the event $\{Z_N(\widehat \tau) \in \,
\angle_N\}$, the process $Z_N$ hits the ball of radius $\ell_N$
centered at the origin before reaching the ball of radius $ 2N$
centered at the origin: $\{Z_N(\widehat \tau) \in \, \angle_N \}
\subset \{ H(B(\ell_N)) < H(B(2N)) \}$, so that
\begin{equation*}
\bs Q_{(l,m)} \big[ \, Z_N(\widehat \tau) \in \, \angle_N \big]
\;\le\;
\widehat{\bs Q}_{(l,m)} \big[ \, H(B(\ell_N)) < H(B(2N)) \, \big]\;.
\end{equation*}
By \cite[Exercice 1.6.8]{law91}, this later quantity is bounded by
\begin{equation*}
\frac{\log 2N - \log |(l,m)| + C \ell_N^{-1}}{\log 2N - \log \ell_N}
\end{equation*}
for some finite constant independent of $N$.  This proves \eqref{g16b}
because $|(l,m)|\ge \delta N$ and $\ell_N \ll N$.

We return to \eqref{g17b}. If $X_N(\tau)\in \partial_N$, the
probability vanishes. We may therefore insert inside the expectation
the indicator of the set $X_N(\tau)\not\in \partial_N$ It is also
clear that $\sigma_N(t)$ does not change before time $\tau$. Hence, by
\eqref{g16b}, \eqref{g17b} is bounded by
\begin{equation*}
\begin{aligned}
& \bs E_{(1,j,k)} \Big[ \bs 1\{ X_N(\tau) \in \bb L^+(\ell_N) \} \,
\bs P_{X_N(\tau)} \big[ H(\Sigma) < H(\partial_N) \big]\,
\Big] \;+\; o_N(1) \\
&\quad \le\; \max_{m \ge \ell_N } \bs P_{(1, 0,m)} 
\big[ H(\Sigma) < H(\partial_N) \big] \;+\; o_N(1) \;,
\end{aligned}
\end{equation*}
where $\bb L^+(r) = \{(\sigma, 0,l): \sigma = \pm 1 \,,\, l \ge r\}$,
$o_N(1)$ converges to $0$ as $N\to\infty$, uniformly over all
$(j,k)\in \bb D_N$, $j>\delta N$, and $\ell_N$ is a sequence such that
$\ell_N\ll N$. Hence, up to this point, we proved that
\begin{equation}
\label{g18b}
\max_{\substack{(j,k) \in \bb D_N \\ j>\delta N}}
\bs P_{(1,j,k)} \big[ H(\Sigma) < H(\partial_N) \big] 
\; \le\; \max_{m \ge \ell_N } \bs P_{(1, 0,m)} 
\big[ H(\Sigma) < H(\partial_N) \big] \;+\; o_N(1)\;, 
\end{equation}
where $o_N(1)$ converges to $0$ as $N\to\infty$, and $\ell_N$ is a
sequence such that $\ell_N\ll N$. \smallskip

It remains to estimate the probability appearing in the previous
formula.  If $m>p'$, starting from $(1, 0,m)$, in $p'$ jumps the
process $X_N(t)$ can not hit $\Sigma$. Hence, if $\tau(k)$ stands for
the time of the $k$-th jump, by the strong Markov property,
\begin{align*}
& \bs P_{(1, 0,m)} \big[ H(\Sigma) < H(\partial_N) \big] \;=\;
\bs P_{(1, 0,m)} \big[ H(\partial_N) > \tau(p')  \,,\, 
H(\Sigma) < H(\partial_N) \big] \\
&\quad \;=\;
\bs E_{(1, 0,m)} \Big[ \, \bs 1\{H(\partial_N) > \tau(p')\}  \,
\bs P_{X_N(\tau(p'))}  \big[ H(\Sigma) < H(\partial_N) \big]\, \Big]\;.
\end{align*}
Let $\varrho = \bs P_{(1, 0,m)} [ H(\partial_N) > \tau(p')] = \bs
P_{(-1, 0,m)} [ H(\partial_N) > \tau(p')]$. Note that this quantity
does not depend on $m$ in the set $\{(\sigma, 0,b) : \sigma = \pm 1
\,,\, b>p'\}$. Moreover, as $\sum_j r_j>0$, $\varrho<1$. With this
notation, the previous expression is less than or equal to
\begin{equation*}
\varrho\, \max_{\sigma = \pm 1}\,
\max_{a,b}\, \bs P_{(\sigma, a,b)} \big[ H(\Sigma) < H(\partial_N) \big]\;, 
\end{equation*}
where the maximum is carried over all $(a,b)$ which can be attained in
$p'$ jumps from $(0,m)$. This set is contained in the set $\{(c,d) :
-p\le c \le p' \,,\, m-p' \le d\le m+p'\}$.

Recall the definition of the set $\bb L$ introduced just before the
statement of Lemma \ref{l07}.  If $a\ge 1$, the process $X_N(t)$ hits
the set $\bb L$ before the set $\Sigma$. Hence, by Lemma \ref{l07}, if
$q_N$ is an increasing sequence to be defined later, by the strong
Markov property, for $1\le a\le p'$, $b\gg q_N$,
\begin{align*}
& \bs P_{(\sigma, a,b)} \big[ H(\Sigma) < H(\partial_N) \big] \\
&\qquad  \le \; \frac{C_0}{q_N} \;+\;
\bs P_{(\sigma, a,b)} \big[ X^2_N(H(\bb L)) \ge  b - p'q_N \,,\, H(\Sigma) <
H(\partial_N) \big] \\
&\qquad  \le\; \frac{C_0}{q_N} \;+\; \max_{b'\ge b - p'q_N}
\bs P_{(\sigma, 0,b')} \big[ H(\Sigma) < H(\partial_N) \big]\;.
\end{align*}

On the other hand, if $a\le -1$, let $\bb C_d = \{(\sigma, c,d) :
\sigma = \pm 1 \,,\, -p\le c<0 \}$. In this case, starting from
$(a,b)$, in $p'$ jumps the process $X_N(t)$ may hit the set $\bb
L$. Hence, by the strong Markov property, for $a<0$, $b>np'$, $\bs
P_{(\sigma, a,b)} \big[ H(\bb C_{b-np'}) < H(\bb L) \wedge
H(\partial_N) \big] \le \varrho^n_1$ for some
$\varrho_1<1$. Therefore, by the strong Markov property, for $a<0$ and
$b\gg q_N$,
\begin{align*}
& \bs P_{(\sigma, a,b)} \big[ H(\Sigma) < H(\partial_N) \big] \\
&\qquad  \le \;  
\bs P_{(\sigma, a,b)} \big[ H(\bb L) \wedge H(\partial_N)  < 
H(\bb C_{b-p'q_N}) \,,\,  H(\Sigma) < H(\partial_N) \big] \;+\; \varrho^{q_N}_1 \\
&\qquad \le\; \max_{\sigma'=\pm 1}\, 
\max_{b'\ge b - p'q_N} \bs P_{(\sigma', 0,b')} \big[H(\Sigma) < H(\partial_N) \big] 
\;+\; \varrho^{q_N}_1 \;.
\end{align*}

Let 
\begin{equation*}
T_N(b) \;=\; \max_{\sigma=\pm 1}\, \max_{c\ge b } \bs P_{(\sigma,
  0,c)} \big[H(\Sigma) < H(\partial_N) \big] \;. 
\end{equation*}
Note that the first term appearing on the right-hand side of
\eqref{g18b} is $T_N(\ell_N)$ because the probability does not depend
on the value of $\sigma$.  By the previous arguments, there exists a
finite constant $C_0$ such that for all $b\gg q_N$,
\begin{equation*}
T_N(b) \;\le\;  \varrho\, \Big\{ T_N(b-p'q_N) \;+\;
\frac{C_0}{q_N}\Big\}
\end{equation*}
because $\varrho^{q}_1 \le 1/q$ for all $q$ large enough.  Iterating
this inequality $r_N$ times, we get that for all $b\gg q_Nr_N$,
\begin{equation*}
T_N(b) \;\le\;  \frac{C_0}{q_N} \{ \varrho \,+\, \cdots \,+\,
\varrho^{r_N} \} \;+\; \varrho^{r_N} \;\le\; 
\frac{\varrho}{1-\varrho}\, \frac{C_0}{q_N}
\;+\; \varrho^{r_N} \;. 
\end{equation*}
In view of \eqref{g18b} and of the previous estimate, to complete the
proof of the lemma, it remains to choose sequences $q_N$, $r_N$ such
that $q_N\to \infty$, $r_N\to\infty$, $r_N\, q_N \ll \ell_N$.
\end{proof}

\begin{lemma}
\label{l02}
Assume that $\sum_j r_j>0$. Then, for every $\delta>0$,
\begin{equation*}
\lim_{N\to\infty} \max_{\substack{(j,k) \in \bb D_N \\ j>\delta N}}\,
\big|\, \varphi^{(1)}_N(1,j,k) \, \big| \; =\; 0\;. 
\end{equation*}
\end{lemma}

\begin{proof}
Fix $(j,k) \in \bb D_N$ such that $0<j<k$.
Denote by $D_N$ the diagonal, $D_N = \{(\sigma, l,l+1): \sigma = \pm 1
\,,\, -p\le l <N-1\}$, and by $D_{N,p}$ its restriction to
$\Lambda^*_p$, $D_{N,p} = \{(\sigma, l,l+1): \sigma = \pm 1 \,,\,
-p\le l \le 0\}$.  By Lemma \ref{ll03}, there exists a finite constant
$C_0$ such that for all $(l,m)\in\bb D_N$,
\begin{equation*}
|F_N (\sigma, l.m)| \,\le\, 
\frac {C_0}{N^2}\, \bs 1\{D_N\} (\sigma,
l,m) \;+\; C_0 \, \bs 1\{D_{N,p}\} (\sigma, l,m) \;.
\end{equation*}
Therefore, recalling that $H_N$ was defined as the hitting time of the
boundary $\partial \Xi_N$,
\begin{equation}
\label{g12}
\begin{aligned}
\big|\, \varphi^{(1)}_N(1,j,k) \, \big| \, 
& \le\, \frac {C_0}{N^2}\, 
\bs E_{(1,j,k)} \Big[ \int_0^{H_N}  
\bs 1\{D_N\setminus D_{N,p}\} (X_N(s)) \, ds \Big] \\
\; & +\; C_0 \, \bs E_{(1,j,k)} \Big[ \int_0^{H_N}  
\bs 1\{D_{N,p}\} (X_N(s)) \, ds \Big] \;.
\end{aligned}
\end{equation}

We claim that there exists a finite constant $C_0$ such that
\begin{equation}
\label{g13}
\max_{\sigma = \pm 1} \max_{\substack{(j,k) \in \bb D_N \\ 0<j<k}}\, 
\bs E_{(\sigma,j,k)} \Big[ \int_0^{H_N}  \bs 1\{D_N\setminus D_{N,p}\} 
(X_N(s)) \, ds \Big] \;\le\;  C_0 \,N \;.
\end{equation}
To bound this expectation, let $R_N = \{(\sigma, 0,m) : \sigma = \pm 1
\,,\, 2\le m \le N-1\}$, and denote by $G_N$ the hitting time of the
set $R_N \cup \partial \, \Xi_N$. Note that starting from $(j,k)$,
$0<j<k$, only the component $\big\{ (\sigma, l,N) : -p\le l<N-1
\big\}$ of the set $\partial\, \Xi_N$ can be attained before the set
$R_N$. Moreover, before $G_N$ the process $X_N(t)$ behaves as a
symmetric random walk.

Rewrite the expectation in \eqref{g13} as
\begin{equation}
\label{g14}
\bs E_{(\sigma,j,k)} \Big[ \int_0^{G_N}  \bs 1\{D_N\setminus D_{N,p}\} (X_N(s)) \, ds
\Big] \;+\; 
\bs E_{(\sigma,j,k)} \Big[ \int_{G_N}^{H_N} \bs 1\{D_N\setminus D_{N,p}\} (X_N(s)) \, ds
\Big]\;.
\end{equation}
Since before time $G_N$ the process $X_N(t)$ evolves as a symmetric
random walk, the first expectation can be computed. It is equal to $
j(N-k)/(N-1) \le C_0\, N$. By the strong Markov property, the second
expectation is bounded above by
\begin{equation*}
\max_{2\le m<N}
\bs E_{(\sigma,0,m)} \Big[ \int_{0}^{H_N}  \bs 1\{D_N\setminus D_{N,p}\} (X_N(s)) \, ds
\Big]  \;.
\end{equation*}

Denote by $\Upsilon_N$ the previous expression and by $G^+_N$ the
return time to $R_N \cup \partial \, \Xi_N$. By the strong Markov
property, the previous expectation is bounded above by
\begin{align*}
\bs E_{(\sigma,0,m)} \Big[ \int_{0}^{G^+_N}  \bs 1\{D_N\setminus D_{N,p}\} (X_N(s)) \, ds
\Big]  \;+\; \Upsilon_N\, \max_{0\le m'<N-1}
\bs P_{(\sigma,0,m')} \big[ \, G^+_N < H_N \, \big] \;.
\end{align*}
The first term vanishes unless the first jump of $X_N(s)$ is to
$(\sigma,1,m)$. Suppose that this happens. Starting from
$(\sigma,1,m)$, up to time $G^+_N$, $X_N(s)$ behaves as a symmetric
random walk. Hence, by explicit formula for the first term in
\eqref{g14}, the expectation is equal to $(N-m)/(N-1)\le 1$. Hence,
\begin{equation*}
\Upsilon_N  \;\le\;  \max_{0\le m'<N-1} \frac 1{\bs
  P_{(\sigma,0,m')} \big[ \, H_N < G^+_N \, \big] } \;\cdot
\end{equation*}
As $\sum_j r_j>0$, $P_{(\sigma,0,m')} [ \, H_N < G^+_N \, ]$ is
bounded below by the probability that the process jumps to a site
$(\sigma, l,m')$ such that $r_l>0$ and then hits the set $\partial\,
\Xi_N$. Hence, there exists a positive constant $c_0$ such that
$P_{(\sigma,0,m')} [ \, H_N < G^+_N \, ] \ge c_0$ for all $2\le m'\le
N-1$. This proves that $\Upsilon_N \le C_0$. Assertion \eqref{g13} follows
from this bound and the estimate for the first term in \eqref{g14}. 

We turn to the second term in \eqref{g12}. Recall the notation
introduced just before Lemma \ref{l08}. Since the integrand vanishes
before hitting the set $D_{N,p}$ and since the set $\Sigma$ is
attained before $D_{N,p}$, for $j>\delta N$
\begin{align*}
& \bs E_{(1,j,k)} \Big[ \int_0^{H_N}  \bs 1\{D_{N,p}\} (X_N(s)) \, ds
\Big]  \\
&\quad \;=\; \bs E_{(1,j,k)} \Big[ \bs 1\{ H(\Sigma) <
H(\partial_N)\} \, \int_{H(D_{N,p})}^{H_N}  \bs 1\{D_{N,p}\} (X_N(s)) \, ds
\Big] \;.
\end{align*}
Applying the strong Markov property twice, we bound this expression by
\begin{equation*}
\bs P_{(1,j,k)} \big[ \, H(\Sigma) < H(\partial_N) \, \big] \,
\max_{(\sigma, a,b) \in D_{N,p}} \bs E_{(\sigma,a,b)} 
\Big[ \int_0^{H_N}  \bs 1\{D_{N,p}\} (X_N(s)) \, ds \Big]\;.
\end{equation*}
By Lemma \ref{l08} the first term vanishes as $N\to\infty$, uniformly
over $(j,k)\in \bb D_N$, $j>\delta N$.

It remains to show that there exists a finite constant $C_0$ such that
\begin{equation}
\label{g15}
\max_{(\sigma,j,k) \in D_{N,p}}\, 
\bs E_{(\sigma,j,k)} \Big[ \int_0^{H_N}  
\bs 1\{D_{N,p}\} (X_N(s)) \, ds \Big] 
\;\le\;  C_0 \;.
\end{equation}
Denote this expression by $\Upsilon_N$, and by $J^+_N$ the return time
to $D_{N,p}$. For $(\sigma,j,k) \in D_{N,p}$, the previous expectation is
less than or equal to
\begin{equation*}
C_0 \;+\; \Upsilon_N \, \bs P_{(\sigma,j,k)} \big[ \, J^+_N < H_N \,
\big] \;. 
\end{equation*}
As in the first part of the proof, since $\sum_j r_j>0$, the process
hits $\partial\,\Xi_N$ before returning to $D_{N,p}$ with a
probability bounded below by a strictly positive constant independent
of $N$: $\min_{(\sigma, j,k) \in D_{N,p}}\, \bs P_{(\sigma,j,k)} [ \, H_N <
J^+_N \, ] \ge c_0>0$. Therefore, $\Upsilon_N \le C_0$. This completes
the proof of assertion \eqref{g15} and the one of the lemma.
\end{proof}

\begin{lemma}
\label{l06}
Assume that $\sum_j r_j>0$. Then, for every $\delta>0$,
\begin{equation*}
\lim_{N\to\infty} \max_{\substack{(j,k) \in \bb D_N \\ j>\delta N}}\, 
\big|\, \varphi^{(2)}_N(1,j,k) \, \big| \; =\;  0\;.
\end{equation*}
\end{lemma}

\begin{proof}
Fix $\delta>0$ and $(j,k) \in \bb D_N$ such that $j>\delta N$. 
Recall the notation introduced just before Lemma \ref{l08}.
In view of the definition of $b_N$, given in \eqref{g08}, \eqref{g07},
\eqref{g06}, 
\begin{equation*}
| \varphi^{(2)}_N(1,j,k) |\; \le\; \bs P_{(1,j,k)} \big[ \,
H(\Sigma) < H(\partial_N)  \, \big]\;.
\end{equation*}
The assertion of the lemma follows from Lemma \ref{l08}.
\end{proof}

\begin{proof}[Proof of Theorem \ref{mt0}]
The proof is straightforward.  It is enough to prove the result for
continuous functions with compact support in $(0,1)$.  Fix such a
function $G$ and let $\delta>0$ such that the support of $G$ is
contained in $[\delta, 1-\delta]$. By Schwarz inequality and by
\eqref{g10}, the square of the expectation appearing in the statement
of the theorem is bounded above by 
\begin{equation*}
C(G)\, \Big(\frac 1{N} \sum_{k=1}^{N-1} 
\big|\, \rho_N(k) - \bar u(k/N)\,\big| \Big)^2
\;+\; \frac {C(G)}{N^2} \sum_{j,k=1}^{N-1} 
G(j/N)\, G(k/N) \, \varphi_N(1,j,k) \;,
\end{equation*}
where $\varphi_N$ has been introduced in \eqref{g10} and $C(G)$ a
finite constant which depends only on $G$.  By Lemmata \ref{ll03},
\ref{l02} and \ref{l06} this expression vanishes as $N\to\infty$.
\end{proof}

\begin{remark}
\label{rm0bis}
Assume that $\sum_{j\in\Lambda^*_p} r_j=0$ and
$\sum_{j,k\in\Lambda^*_p} a_{j,k} >0$. The proof that the correlations
vanish, presented in Lemmata \ref{l02} and \ref{l06}, requires a new
argument based on the following observation. Under the conditions of
this remark, the boundary $\partial\,\Xi_N$ of the set $\Xi_N$ is
reduced to the set
\begin{equation*}
\big\{ (\sigma, k,k) :
\sigma = \pm 1 \,,\, -p\le k\le 0 \big\} 
\,\cup\,
\big\{ (\sigma, k,N) :
\sigma = \pm 1 \,,\, -p\le k<N-1 \big\}\;.
\end{equation*}
To prove that the correlations vanish, one has to show that by the
time the process $X_N(t)$ hits the set $\{ (\sigma, k,k) : \sigma =
\pm 1 \,,\, -p\le k\le 0 \}$ its coordinate $\sigma$ has equilibrated
and takes the value $\pm 1$ with probability close to $1/2$.
\end{remark}

\section{Proof of Theorem \ref{mt1}}
\label{sec02}

The proof of Theorem \ref{mt1} is based on a graphical construction of
the dynamics through independent Poisson point processes. 

Recall the definition of the rates $A$, $B$ introduced in \eqref{f12},
that $\Omega_p = \{0,1\}^{\{1, \dots, p-1\}}$, and that
$\lambda (0,\xi) = c (0,\xi)-A$, $\lambda (1,\xi) = c (1,\xi)-B$,
$\xi\in \Omega_p$. Further, recall that we assume 
\begin{equation*}
(p-1) \sum_{\xi\in\Omega_p} \{\, \lambda(0,\xi) + \lambda(1,\xi)\, \} 
\;<\; A \,+\, B \;. 
\end{equation*}

The left boundary generator can be rewritten as 
\begin{align*}
(L_{l} f)(\eta) \; &=\; A\, [f(T^1\eta)-f(\eta)] \;+\; B\,
[f(T^0\eta)-f(\eta)] \\
& +\; \sum_{a=0}^1 \sum_{\xi\in \Omega_p} \lambda(a,\xi) \,
\bs 1 \{\Pi_p \eta = (a,\xi)\}\, [f(T^{1-a}\eta)-f(\eta)]\;,
\end{align*}
provided $\Pi_p : \Omega_N \to \Omega_p^\star := \{0,1\}^{\{1, \dots,
  p\}}$ represents the projection on the first $p$ coordinates:
$(\Pi_p\eta)_k = \eta_k$, $1\le k\le p$. Similarly, the right boundary
generator can be expressed as
\begin{equation*}
(L_{r,N} f)(\eta) \; =\; \beta\, [f(S^1\eta)-f(\eta)]
\;+\; (1-\beta)\, [f(S^0\eta)-f(\eta)] \;,
\end{equation*}
where
\begin{equation*}
(S^a\eta)_k \;=\;
\begin{cases}
a & \text{if $k=N-1$,} \\
\eta_k & \text{otherwise.}
\end{cases}
\end{equation*}

\subsection{Graphical construction}
\label{sec:GraphicalConstruction}
Let $P:=2^{p-1}=\abs{\Omega_p}$.  We present in this subsection a
graphical construction of the dynamics based on $N+2P+2$ independent
Poisson point processes defined on $\bb R$.
\begin{itemize}
\item [--] $(N-2)$ processes $\mf N_{i,i+1}(t)$, $1\le i\le N-2$, with
  rate $1$.
\item [--] $2$ processes $\mf N^{+,l}(t), \; \mf N^{-,l}(t)$ with
  rates $A$, $B$, respectively, representing creation and annihilation
  of particles at site $1$, regardless of the boundary condition.
\item [--] $2P$ processes $\mf N_{(a,\xi)}(t)$, $a=0$, $1$,
  $\xi\in\Omega_p$, with rates $\lambda(a,\xi)$ to take into
  account the influence of the boundary in the creation and
  annihilation of particles at site $1$.
\item [--] $2$ processes $\mf N^{+,r}(t), \; \mf N^{-,r}(t)$, with
  respective rates $\beta$ and $1-\beta$, to trigger creation and
  annihilation of particles at site $N-1$.
\end{itemize}

Place arrows and daggers on $\{1, \dots, N-1\} \times \bb R$ as
follows. Whenever the process $\mf N_{i,i+1}(t)$ jumps, place a
two-sided arrow over the edge $(i,i+1)$ at the time of the jump to
indicate that at this time the occupation variables $\eta_i$,
$\eta_{i+1}$ are exchanged. Analogously, each time the process $\mf
N_{(a,\xi)}(t)$ jumps, place a dagger labeled $(a,\xi)$ over the
vertex $1$. Each time $\mf N^{\pm,l}(t)$ jumps, place a dagger labeled
$\pm$ over the vertex $1$.  Finally, each time $\mf N^{\pm,r}(t)$
jumps, place a dagger labeled $\pm$ over the vertex $N-1$.

Fix a configuration $\zeta\in\Omega_N$ and a time $t_0\in \bb
R$. Define a path $\eta(t)$, $t\ge t_0$, based on the configuration
$\zeta$ and on the arrows and daggers as follows.  By independence, we
may exclude the event that two of those processes jump
simultaneously. Let $\tau_1>t_0$ be the first time a mark (arrow or
dagger) is found after time $t_0$. Set $\eta(t) = \zeta$ for any
$t\in[t_0,\tau_1)$. If the first mark is an arrow labeled $(i,i+1)$,
set $\eta(\tau_1) = \sigma^{i,i+1} \eta(\tau_1-)$. If the mark is a
dagger labeled $(a,\xi)$, set $\eta(\tau_1) = T^{a} \eta(\tau_1-)$ if
$\Pi_p\eta(\tau_1-) = (a,\xi)$. Otherwise, let $\eta(\tau_1) =
\eta(\tau_1-)$. Finally, if the mark is a dagger on site $1$,
resp. $N-1$, labeled $\pm$, set $\eta(\tau_1) = T^{[1 \pm 1]/2}
\eta(\tau_1-)$, resp. $\eta(\tau_1) = S^{[1 \pm 1]/2} \eta(\tau_1-)$.

At this point, the path $\eta$ is defined on the segment
$[t_0,\tau_1]$.  By repeating the previous construction on each
time-interval between two consecutive jumps of the Poisson point
processes, we produce a trajectory $(\eta(t) : t\ge t_0)$. We leave
the reader to check that $\eta(t)$ evolves as a continuous-time Markov
chain, started from $\zeta$, whose generator is the operator $L_N$
introduced in \eqref{f13}.

\subsection{Dual Process}
\label{sec:devices}

To determine whether site $1$ is occupied or not at time $t=0$ we have
to examine the evolution backward in time. This investigation, called
the revealment process, evolves as follows.

Let mark mean an arrow or a dagger. To know the value of $\eta_1(0)$
we have to examine the past evolution.  Denote by $\tau_1<0$ the time
of the last mark involving site $1$ before $t=0$. By the graphical
construction, the value of $\eta_1$ does not change in the time
interval $[\tau_1, 0]$.

Suppose that the mark at time $\tau_1$ is an arrow between $1$ and
$2$. In order to determine if site $1$ is occupied at time $0$ we need
to know if site $2$ is occupied at time $\tau_1-$. The arrows are thus
acting as a stirring dynamics in the revealment process. Each time an
arrow is found, the site whose value has to be determined changes.

If the mark at time $\tau_1$ is a dagger labeled $+$ at site $1$,
$\eta_1(0)=\eta_1(\tau_1)=1$, and we do not need to proceed
further. Analogously, daggers labeled $-$ or $+$ at sites $1$, $N-1$
reveal the value of the occupation variables at these sites at the
time the mark appears. Hence, these marks act an annihilation
mechanism.

Suppose that the mark at time $\tau_1$ is a dagger labeled
$(a,\xi)$. To determine whether site $1$ is occupied at time $0$ we
need to know the values of $\eta_1(\tau_1-), \dots,
\eta_p(\tau_1-)$. Indeed, if $\Pi_p \eta(\tau_1-)=(a,\xi)$, $\eta_1(0)
= \eta_1(\tau_1)=1-a$, otherwise, $\eta_1(0) = \eta_1(\tau_1) =
\eta_1(\tau_1-)$. Hence, marks labeled $(a,\xi)$ act as branching
events in the revealment process.

It follows from this informal description that to determine the value
at time $0$ of site $1$, we may be forced to find the values of the
occupation variables of a larger subset $\ms A$ of $\Lambda_N$ at a
certain time $t<0$. 

Suppose that we need to determine the values of the occupation
variables of the set $\ms A \subset \Lambda_N$ at time $t<0$. Let
$\tau<t$ be the first [backward in time] mark of one of the Poisson
processes: there is a mark at time $\tau$ and there are no marks in
the time interval $(\tau, t]$.  Suppose that the mark at time $\tau$
is

\begin{itemize}
\item[(a)] an arrow between $i$ and $i+1$;
\item[(b)] a dagger labeled $\pm$ at site $1$;
\item[(c)] a dagger labeled $\pm$ at site $N-1$;
\item[(d)] a dagger labeled $(a,\xi)$ at site $1$.
\end{itemize}

\noindent Then, to determine the values of the occupation variables in
the set $\ms A$ at time $\tau$ (and thus at time $t$), we need to find
the values of the occupation variables in the set

\begin{itemize}
\item[(a)] $\sigma^{i,i+1}\ms A$, defined below in \eqref{eq:sigmaA};
\item[(b)] $\ms A\setminus \{1\}$;
\item[(c)] $\ms A\setminus \{N-1\}$;
\item[(d)] $\ms A \cup \{1, \dots, p\}$ if $1\in \ms A$, and $\ms A$
  otherwise 
\end{itemize}

\noindent at time $\tau-$. Since independent Poisson processes run
backward in time are still independent Poisson processes, this
evolution corresponds to a Markov process taking values in $\Xi_N$,
the set of subsets of $\Lambda_N$, whose generator $\mf L_N$ is given
by
\begin{equation*}
\mf L_N \;=\; \mf L_{l} \;+\; \mf L_{0,N} \;+\; \mf L_{r,N} \;,
\end{equation*}
where
\begin{equation*}
(\mf L_{0,N} f)(\ms A) \; =\; \sum_{i=1}^{N-2} [f(\sigma^{i,i+1} \ms A) -
f(\ms A)] \;;
\end{equation*}
\begin{align*}
(\mf L_{l} f)(\ms A)  & \;=\; (A+B)  \, \bs 1\{1\in \ms A\}\,
\pa{f(\ms A \setminus \{1\}) - f(\ms A)} \\
& \;+\; \sum_{\xi\in\Omega_p} \lambda(\xi) \, \bs 1\{1\in \ms A\}\,
\pa{f(\ms A \cup \{1,\dots,p\}) - f(\ms A)}\;;
\end{align*}
\begin{equation*}
(\mf L_{r,N} f)(\ms A) = f(\ms A \setminus \{N-1\}) - f(\ms A) \;.
\end{equation*}
In these formulae, $\lambda(\xi) = \lambda(0,\xi) + \lambda(1,\xi)$, and
\begin{equation}
\label{eq:sigmaA}
\sigma^{i,i+1}\ms A=\begin{cases}
\ms A \cup\{i+1\}\setminus \{i\} &\mbox{ if }i\in \ms A, i+1\notin \ms A\\
\ms A \cup\{i\}\setminus \{i+1\} & \mbox{ if }i\notin \ms A, i+1\in \ms A\\
\ms A &\mbox{ otherwise }.
\end{cases}
\end{equation}

Denote by $\ms A(s)$ the $\Xi_N $-valued process whose generator is
$\mf L_N$ and which starts from $\{1\}$. If $\ms A(s)$ hits the empty
set at some time $T>0$ due to the annihilations, this means that we
can reconstruct the value of site $1$ at time $0$ only from the
Poisson point processes in the time interval $[-T,0]$, and with no
information on the configuration at time $-T$, $\eta(-T)$.

On the other hand, it should be verisimilar that if the number of
daggers labeled $\pm$ is much larger that the number of daggers
labeled $(a,\xi)$, that is, if the rates $\lambda(a,\xi)$ are much
smaller than $A+B$, the process $\ms A(s)$ should attain the empty
set. The next lemmata show that this is indeed the case. \smallskip

Let
\begin{equation*}
T\;=\; \inf \{s>0 : \ms A(s)=\varnothing \}\;.
\end{equation*}
It is clear that for any $s>0$, the value of $\eta_1(0)$ can be
recovered from the configuration $\eta(-s)$ and from the Poisson marks
in the interval $[-s, 0]$. The next lemma asserts that $\eta_1(0)$ can
be obtained only from the Poisson marks in the interval $[-T, 0]$.

\begin{lemma}
\label{cl01}
Assume that $T<\infty$. The value of $\eta_1(0)$ can be recovered from
the marks in the time interval $[-T, 0]$ of the $N+ 2(P+1)$ Poisson
point processes $\mf N$ introduced in the beginning of this section.
\end{lemma}

\begin{proof}
Let $\Xi'_N = \{0,1,u\}^{\Lambda_N}$, where $u$ stands for unknown.
Denote by $\zeta$ the configurations of $\Xi'_N$.  We first construct,
from the marks of the Poisson point processes $\mf N(t)$ on $[-T,0]$,
a $\Xi'_N$-valued evolution $\zeta (s)$ on the time interval $[(-T)-, 0]$ in
which the set $B(s) = \{k\in \Lambda_N : \zeta_k(s) \not = u\}$ represents
the sites whose occupation variables can be determined by the
Poisson point processes only.

Let $\zeta_k([-T]-) = u$ for all $k\in \Lambda_N$. By definition of
the evolution of $\ms A(s)$, $T$ corresponds to a mark of one of the
Poisson point processes $\mf N^{\pm,l}$, $\mf N^{\pm,r}$. We define
$\zeta(-T)$ as follows. If it is a mark from $\mf N^{\pm,l}$ we set
$\zeta_1(-T)=[1\pm 1]/2$ and $\zeta_k(-T)=u$ for $k\not
=1$. Analogously, if it is a mark from $\mf N^{\pm,r}$ we set
$\zeta_{N-1}(-T)=[1\pm 1]/2$ and $\zeta_k(-T)=u$ for $k\not =N-1$.

Denote by $-T = \tau_0 < \tau_1 < \cdots < \tau_M<0<\tau_{M+1}$ the
successive times at which a dagger of type $\pm$ occurs at site $1$ or
$N-1$. If $\tau_j$ corresponds to a mark from $\mf N^{\pm,l}$ we set
$\zeta_1(\tau_j)=[1\pm 1]/2$ and we leave the other values
unchanged. We proceed analogously if $\tau_j$ corresponds to a mark
from $\mf N^{\pm,r}$. There are (almost surely) a finite number of
such times because $T<\infty$ by assumption.

In the intervals $(\tau_j, \tau_{j+1})$, holes, particles and unknowns
exchange their positions according to the marks of $\mf
N_{i,i+1}(t)$. Each time $\sigma$ a dagger of type $\lambda(a,\xi)$ is
found, if $(\zeta_1(\sigma-), \dots, \zeta_p(\sigma-)) = (a, \xi)$, we
update the configuration accordingly. Otherwise, we leave the
configuration unchanged. This completes the description of the
evolution of the process $\zeta(s)$.

We claim that 
\begin{equation}
\label{g20}
B(s) \,\supset\, \ms A([-s]-) \quad\text{for all}\quad
-T \,\le\, s \,\le\, 0\;.
\end{equation}
The left limit $(-s)-$ in $\ms A([-s]-)$ appears because by convention the
processes $\zeta(s)$ and $\ms A(s)$ are both right-continuous and the
latter one is run backwards in time.

We prove this claim by recurrence. By construction, $B([-T]-) = \ms
A(T) = \varnothing$ and $B(-T) = \ms A(T-) = \{1\}$ or $\{N-1\}$,
depending on the mark occurring for $\ms A$ at time $T$. It is clear
that if $B(\tau-) \supset \ms A(-\tau)$, where $\tau \in [-T, 0)$ is
an arrow of type $\mf N_{i,i+1}$ or a mark of type $\mf N^{\pm,l}$,
$\mf N^{\pm,r}$, then $B(\tau) \supset \ms A([-\tau]-)$. Observe that
the inclusion may be strict. For example, if $\tau \in [-T, 0)$ is a
mark of type $\mf N^{+,l}$ and $\ms A([-\tau]-)$ does not contain
$1$. This mark permits to determine the value of site $1$ at time
$\tau$, so that $B(\tau) \ni 1$ but $\ms A([-\tau]-) \not \ni 1$.

Similarly, suppose that $B(\tau-) \supset \ms A(-\tau)$ and that $\tau
\in (-T, 0)$ is a mark of type $\mf N_{(a,\xi)}$. If $1$ belongs to
$\ms A([-\tau]-)$, then $\ms A(-\tau)$ contains $\{1, \dots, p\}$ and
so does $B(\tau-)$ because $B(\tau-) \supset \ms A(-\tau)$. Hence, all
information to update site $1$ is available at time $\tau-$ and $1\in
B(\tau) =B(\tau-)$. Since $\ms A([-\tau]-)$ is contained in $\ms
A(-\tau)$ [it can be strictly contained because some points $m\in \{2,
\dots, p\}$ may not belong to $\ms A([-\tau]-)$], $B(\tau) \supset \ms
A([-\tau]-)$. 

On the other hand, if $1$ does not belong to $\ms A([-\tau]-)$, then
$\ms A([-\tau]-) = \ms A(-\tau)$, while $B(\tau) \supset
B(\tau-)$. [This relation may be strict because it might happen that
$1\not \in B(\tau-)$ and there might be enough information to
determine the value of site $1$ at time $\tau$.] Thus $B(\tau) \supset
B(\tau-) \supset \ms A(-\tau) = \ms A([-\tau]-)$. This proves claim
\eqref{g20}. 

Since $\ms A(0)=\ms A(0-)=\{1\}$, by \eqref{g20}, $B(0)\ni 1$, which
proves the lemma.
\end{proof}

Denote by $\bb Q_N$ the probability measure on $D(\bb R_+, \Xi_N)$
induced by the process $\ms A(s)$ starting from $\{1\}$. Expectation
with respect to $\bb Q_N$ is represented by $\bb Q_N$ as well.

Denote by $C(s)$ the total number of particles created up to time
$s$. The next lemma provides a bound for the total number of particles
created up to the absorbing time $T$.

\begin{lemma}
\label{cl02}
Let $\lambda = \sum_{\xi\in\Omega_p} \{\lambda(0,\xi) +
\lambda(1,\xi)\}$. Then, 
\begin{equation*}
\bb Q_N \, [C(T)] \;\le\; \frac{(p-1) \lambda}
{A+B - (p-1) \lambda}\;\cdot
\end{equation*}
\end{lemma}

\begin{proof}
Let $X(t)$ be a continuous-time random walk on $\bb Z$ which jumps
from $k$ to $k-1$, resp. $k+p-1$, at rate $A+B$,
resp. $\lambda$. Suppose that $X(0)=1$, and let $T_0$ be the first
time the random walk hits the origin. As $X(t\wedge T_0) + [A+B -
(p-1)\, \lambda] \, (t\wedge T_0)$ is an integrable, mean-$1$
martingale,
\begin{equation*}
[A+B - (p-1)\, \lambda] \, E \big[\, t\wedge T_0\, \big] \;=\;
1\;-\; E \big[\, X(t\wedge T_0) \, \big] \;\le\; 1\;.
\end{equation*}
Letting $t\to\infty$ we conclude that $E[T_0]\le 1/(A+B - (p-1)
\lambda)$.

Let $R(s)$ be the total number of jumps to the right of the random
walk $X$ up to time $s$. $R$ is a Poisson process of rate $\lambda$ so
that $R(s) - \lambda\, s$ is a martingale. Hence, $E[R(s\wedge T_0)] =
\lambda\, E[s\wedge T_0]$. Letting $s\to\infty$, we obtain that
\begin{equation*}
E[R(T_0)] \;=\; \lambda\, E[T_0] \;\le\;
\frac{\lambda} {A+B - (p-1) \lambda}\;\cdot
\end{equation*}

Consider the process $\ms A(s)$ associated to the generator $\mf L_N$.
Denote the cardinality of a set $B\in \Xi_N$ by $|B|$. $|\ms A(s)|$
only changes when the set $\ms A(s)$ contains $1$ or $N-1$. The
Poisson daggers at $N-1$ may only decrease the cardinality of the set.
When $\ms A(s)$ contains $1$, Poisson daggers of type $\pm$ appear at
site $1$ at rate $A+B$ and they decrease the cardinality of $\ms A(s)$
by $1$. Analogously, the other daggers appear at site $1$ at rate
$\lambda$ and increase the cardinality by at most $p-1$. This shows
that we may couple $|\ms A(s)|$ with the random walk $X(s)$ in such a
way that $|\ms A(s)|\le X(s)$ and that $C(s) \le (p-1) R(s)$ for all
$0\le s\le T_0$. The assertion of the lemma follows from the bound
obtained in the first part of the proof.
\end{proof}

As the total number of particles created in the process $\ms A(s)$ has
finite expectation, and since these particles are killed at rate $A+B$
when they reach site $1$, the life-span $T_0$ of $\ms A(s)$ can not be
large and the set of sites ever visited by a particle in $\ms A(s)$
can not be large. This is the content of the next two lemmata.

\begin{lemma}
\label{cl03}
For any sequence $\ell_N \to \infty$,
\begin{equation*}
\lim_{N\to \infty}  \bb Q_N \big[\, T > N\, \ell_N\,\big ] \;=\;
0\; .
\end{equation*}
\end{lemma}

\begin{proof}
Fix a sequence $\ell_N\to\infty$, let $m_N=\sqrt{\ell_N}$, and write
\begin{equation*}
\bb Q_N \big[\, T > N\, \ell_N\,\big ] \;\le \;
\bb Q_N \big[\, T > N\, \ell_N\, ,\, C(T) \leq m_N \, \big ]
\;+\; \bb Q_N \big[\, C(T) >  m_N \, \big ] \;.
\end{equation*}
By the Markov inequality and Lemma \ref{cl02}, the second term vanishes as
$N\to\infty$. 

Denote by $T_1$ the lifespan of the particle initially at $1$, and by
$T_k$, $2\leq k\leq C(T)$, the lifespan of the $k$-th particle created
in the process $\ms A(s)$. By lifespan, we mean the difference
$\tau_k-\sigma_k$, where $\sigma_k$, resp. $\tau_k$, represents the
time the $k$-th particle has been created, resp. annihilated. Clearly,
\begin{equation*}
T\;\le\; \sum_{k=1}^{C(T)} T_k\;.
\end{equation*}
Set $T_k=0$ for $k>C(T)$. The first term on the right-hand side of the
penultimate formula is bounded above by
\begin{equation*}
\bb Q_N \Big[\, \sum_{k=1}^{m_N} T_k > N\, \ell_N \, \Big ]
\;\leq\;  \frac{m_N}{N\, \ell_N} \sup_{k\geq 1} \bb Q_N [\, T_k\,] \; .  
\end{equation*}
It remains to show that there exists a finite constant $C_0$ such that
for all $k\ge 1$,
\begin{equation}
\label{2-14}
\bb Q_N [\, T_k\,] \;\leq\;  C_0\, N\; .
\end{equation}

Particles are created at one of the first $p$ sites. After being
created, they perform a symmetric random walk at rate $1$ on
$\Lambda_N$. Each time a particle hits site $1$, resp. $N-1$, it is
destroyed at rate $A+B$, resp. $1$. We overestimate the lifespan by
ignoring the annihilation at the right boundary.
	
Consider a particle performing a rate $1$ random walk on $\Lambda_N$
with reflection at the boundary $N-1$ and annihilated at rate $A+B$ at
site $1$.  Denote by $\bs P_k$ the distribution of this random walk
started from site $k$, and by $\bs E_k$ the corresponding
expectation. Let $T_Y$ be the time this particle is killed at site
$1$, and $Y_t$, $t\le T$ its position at time $t$. By the strong
Markov property, $\bs E_k[T_Y]$ increases with $k$. Hence,
\begin{equation*}
\bb Q_N [\, T_k\,] \;\leq\; \bs E_p[\,T_Y\,]\;.
\end{equation*}

Divide the lifespan $T_Y$ in excursions away from $1$. To keep
notation simple, assume that the random walk $Y$ keeps evolving after
being killed. Denote by $\{t_j : j\ge 1\}$ the successive hitting
times of site $1$: $t_0=0$, and for $i\ge 1$,
\begin{equation*}
t_{i}\;=\; \inf \big\{t>t_{i-1} :  Y(t)=1 \text{ and } Y(t-)\neq 1
\big\}\; .
\end{equation*}
Denote by $u_i$, $i\ge 1$, the time the random walk $Y(t)$ leaves site
$1$ after $t_i$:
\begin{equation*}
u_{i}\;=\; \inf \big\{t>t_{i} :  Y(t)\not =1 \big\}\; ,
\end{equation*}
and set $u_0=0$.  Let $\sigma_i = u_{i} - t_i$,
resp. $s_i=t_i-u_{i-1}$, be duration of the $i$-th sojourn at $1$,
resp. the duration of the $i$-th excursion away from $1$.

Denote by $A_k$ the event ``the particle is annihilated during its
$k$-th sojourn at site $1$''.  With this notation we have that
\begin{equation*}
T_Y  \; \le \; (s_1 + \sigma_1) \;+\;  \sum_{i\ge 2}   
\bs 1\{ A^c_1 \cap \cdots \cap A^c_{i-1}\} \, (s_i \;+\; \sigma_i) \;.
\end{equation*}
By the strong Markov property at time $u_{i-1}$,
\begin{equation*}
\bs E_p \Big[\, \bs 1\{ A^c_1 \cap \cdots \cap A^c_{i-1}\} \, (s_i \;+\;
\sigma_i) \,\Big]\;=\; \bs P_p \Big[\, A^c_1 \cap \cdots \cap A^c_{i-1} \,\Big]
\, \bs E_2 \big[\, s_1 + \sigma_1 \,\big]\;.
\end{equation*}
Since the particle is annihilated at rate $A+B$ and leaves site $1$ at
rate $1$, each time it hits site $1$ it is killed during its sojourn
at $1$ with probability $(A+B)/(A+B+1)$. Thus, by the strong Markov
property, the probability on the right hand side of the previous
displayed equation is equal to $\alpha^{i-1}$, where $\alpha =
1/(A+B+1)$, so that
\begin{equation*}
\bs E_p \big[\, T_Y \,\big]
\; \le \; \bs E_p \big[\,  s_1 + \sigma_1 \,\big]
\;+\;  \frac 1{A+B}\, \bs E_2 \big[\,  s_1 + \sigma_1 \,\big]\;.
\end{equation*}

On the one hand, for any $k\in \Lambda_N$, $\bs E_k [\, \sigma_1 \,]=
1$, On the other hand, $\bs E_2 [\, s_1 \,] \le \bs E_p [\, s_1
\,]$. Since the random walk is reflected at $N-1$, by solving the
elliptic difference equation satisfied by $f(k) = \bs E_k [\, s_1
\,]$, we obtain that $\bs E_p [\, s_1 \,] \le C_0N$ for some finite
constant $C_0$ independent of $N$. This completes the proof
\eqref{2-14} and the one of the lemma.
\end{proof}

The proof of the previous lemma shows that each new particle performs
only a finite number of excursions, where by excursion we mean the
trajectory between the time the particle leaves site $1$ and the time
it returns to $1$. In each excursion the particle visits only a finite
number of sites. This arguments yields that during its lifespan the
process $\ms A(s)$ does not visit many sites. This is the content of
the next result.

\begin{lemma}
\label{2-l2}
For any sequence $\ell_N$ such that $\ell_N \to \infty$, $\ell_N\le N -1$,   
\begin{equation*}
\lim_{N\to \infty}
\bb Q_N \big[ \ms A(s) \ni \ell_N \text{ for some } s\ge 0 \, \big]
\;=\; 0\;.
\end{equation*}
\end{lemma}

\begin{proof}
Fix a sequence $\ell_N$ satisfying the assumptions of the lemma.
Denote by $X_k(s)$ the position at time $s$ of the $k$-th particle
created. Before its creation and after its annihilation we set the
position of the particle to be $0$.  The probability appearing in the
statement of the lemma can be rewritten as
\begin{equation*}
\bb Q_N \Big[ \bigcup_{l=1}^{C(T)} \big\{ X_l(s)  = \ell_N
	\text{ for some } s\ge 0 \big\} \Big]\; .
\end{equation*}
Let $m_N=\sqrt{\ell_N}$. The previous expression is bounded by
\begin{equation*}	
\bb Q_N \Big[ \bigcup_{l=1}^{C(T)} \big\{ X_l(s)  =  \ell_N
\text{ for some } s\ge 0 \big\} \,,\,  C(T) \le m_N \Big]
\;+\;  \frac 1{m_N} \, \bb Q_N [C(T)] \;.
\end{equation*}
By Lemma \ref{cl02}, the second term vanishes as $N\to\infty$. Set
$X_l(s)=0$ for any $l>C(T)$, $s\geq 0$. With this notation, we
can replace $C(T)$ by $m_N$ in the union, to bound the first term in
the previous equation by
\begin{equation*}
\sum_{l=1}^{m_N} \bb Q_N \big[\, X_l(s)  = \ell_N 
\text{ for some } s\ge 0 \,\big]\; .
\end{equation*}
It remains to show that there exists a finite constant $C_0$ such that for
all $l\ge 1$,
\begin{equation}
\label{2-11}
\bb Q_N \big[\, X_l(s)  =  \ell_N 
\text{ for some } s\ge 0 \,\big]
\;\le\; \frac {C_0}{\ell_N}\;\cdot
\end{equation}

To derive \eqref{2-11}, recall the notation introduced in the proof of
the previous lemma. Clearly, for any $l\ge 1$,
\begin{equation*}
\bb Q_N \big[\, X_l(s)  =  \ell_N \text{ for some } s\ge 0 \,\big]
\;\le\; \bs P_p \big[\, Y(s)=\ell_N\mbox{ for some }s\leq T_Y
\,\big]\;.
\end{equation*}
Note that this is not an identity because the $l$-th particle may have
been created at a site $k<p$.

Denote by $U_k$ the event that the particle $Y$ visits the site
$\ell_N$ in the time interval $[u_{k-1}, t_k]$.  Hence,
\begin{equation*}
\big\{ Y^j(s)  = \ell_N \text{ for some } s\ge 0 \, \big\}
\;\subset \; U_1 \cup \bigcup_{i\ge 2} \Big(  A^c_1 \cap \cdots \cap
A^c_{i-1} \cap U_i \Big) \;.
\end{equation*}
By the strong Markov property applied at time $u_{i-1}$,
\begin{equation*}
\bs P_p \big[\, Y^j(s)  = \ell_N \text{ for some } s\ge 0 \, \big]
\; \le \; \bs P_p \big[\, U_1 \, \big]
\;+\; \sum_{i\ge 2}   \bs P_p \big[\, A^c_1 \cap \cdots \cap
A^c_{i-1} \, \big] \, \bs P_2 \big[\, U_1 \, \big] \;.
\end{equation*}
If $Y(0)=k$, the event $U_1$ corresponds to the event that a symmetric
random walk starting from $k$ hits $\ell_N$ before it attains $1$, so
that $\bs P_{k} [U_1] = [k-1]/[\ell_N -1]$. Since the particle is
annihilated with probability $(A+B)/(1+A+B)$ in each of its sojourn at
site $1$, by the strong Markov property, the previous sum is equal to
\begin{equation*}
\frac {p-1}{\ell_N-1} \;+\; \frac {1}{A+B}\, \frac {1}{\ell_N-1}\;\cdot
\end{equation*}
This proves assertion \eqref{2-11}.
\end{proof}

We have now all elements to show that the sequence $\rho_N(1)$
converges. 

\begin{proposition}
\label{cp01}
Suppose that conditions \eqref{fc01} are in force.  The limit
\begin{equation*}
\alpha:=\lim_{N\to \infty} \rho_N(1)
\end{equation*}
exists, and it does not depend on the boundary conditions at $N-1$.
\end{proposition}

\begin{proof}
The proof of this proposition is based on coupling a system evolving
on $\Lambda_N$ with a system evolving on $\Lambda_M$, $1< N<M$ by
using the same Poisson point processes to construct both evolutions.

Let $\{\mf N^{\pm, r, b}(t) : t\in \bb R\}$, $b=1$, $2$, be
independent Poisson point processes, where $\mf N^{+, r, b}$ has rate
$\beta$ and $\mf N^{-, r, b}$ rate $1-\beta$. Use the Poisson point
processes $\mf N_{i,i+1}(t)$, $1\le i< N-1$, $\mf N^{\pm, l}(t)$, $\mf
N_{(a,\xi)} (t)$, $\mf N^{\pm, r, 1}(t)$, $t\in\bb R$, to construct
trajectories of a Markov chain $\eta^N(t)$ whose generator is $L_N$
introduced in \eqref{f13}. Similarly, use the Poisson point processes
$\mf N_{i,i+1}(t)$, $1\le i< M-1$, $\mf N^{\pm, l}(t)$, $\mf N_{(a,\xi)}
(t)$, $\mf N^{\pm, r, 2}(t)$ to construct trajectories of a Markov
chain $\eta^M(t)$ whose generator is $L_M$. Note that on the left
boundary and on $\Lambda_N$ the same Poisson processes are used to
construct both chains.

Denote by $\ms A_N(t)$, $\ms A_M(t)$, $t\ge 0$, the dual processes
evolving according to the Poisson marks described at the beginning of
subsection \ref{sec:devices} with initial condition $\ms A_N(0) = \ms
A_M(0) =\{1\}$. By construction, $\ms A_N(t) = \ms A_M(t)$ for all
$t\ge 0$ if $N-1 \not \in \ms A_N(t)$ for all $t\ge 0$. Hence, since
the value of $\eta^N(0)$ can be recovered from the trajectory $\{\ms
A_N(t) : t\ge 0\}$,
\begin{equation}
\label{f14}
\{\eta^N(0) \not = \eta^M(0)\} \;\subset 
\{\ms A_N(t) \ni N-1 \text{ for some } t\ge 0\} \; .
\end{equation}

Denote by $\widehat{\bb P}_{N,M}$ the probability measure associated
to the Poisson processes $\mf N_{i,i+1}(t)$, $1\le i< M-1$,
$\mf N^{\pm, l}(t)$, $\mf N_{(a,\xi)} (t)$, $\mf N^{\pm, r, a}(t)$.
Expectation with respect to $\widehat{\bb P}_{N,M}$ is represented by
$\widehat{\bb E}_{N,M}$. With this notation,
$\rho_N(1) = E_{\mu_N}[\eta_1] = \widehat{\bb E}_{N,M}
[\eta^N_1(0)]$. Hence,
\begin{equation*}
\big|\, \rho_N(1) - \rho_M(1)\,\big| \;\le\;
\widehat{\bb E}_{N,M} \big[\, \big|\, \eta^N_1(0) - 
\eta^M_1(0) \,\big| \, \big] \;.
\end{equation*}
By \eqref{f14}, this expression is less than or equal to
\begin{equation*}
\widehat{\bb P}_{N,M} \big[\, 
\ms A_N(t) \ni N-1 \text{ for some } t\ge 0 \, \big]
\;=\; \bb Q_N \big[\, 
\ms A(t) \ni N-1 \text{ for some } t\ge 0 \, \big]\;.
\end{equation*}
By Lemma \ref{2-l2} the right-hand side vanishes as $N\to\infty$. This
shows that the sequence $\rho_N(1)$ is Cauchy and therefore converges.

Since the argument relies on the fact that the dual process $\ms
A_N(t)$ reaches $N-1$ with a vanishing probability, the same proof
works if the process $\eta^M(t)$ is defined with any other dynamics at
the right boundary, e.g., reflecting boundary condition.
\end{proof}

In the next result we derive an explicit expression for the density
$\rho_N(k)$ in terms of $\beta$ and $\rho_N(1)$.

\begin{lemma}
\label{2-l5}
For all $k\in\Lambda_N$,
\begin{equation*}
\rho_N(k) \;=\;  \frac{N-k}{N-1} \, \rho_N(1) \;+\; \frac{k-1}{N-1} \,
\beta\;. 
\end{equation*}
\end{lemma}

\begin{proof}
Recall that we denote by $\Delta_N $ the discrete Laplacian:
$(\Delta_N f)(k) = f(k-1) + f(k+1) - 2 f(k)$. Since $\mu_N$ is the
stationary state, $E_{\mu_N} [L_N f] = 0$ for all function
$f:\Omega_N\to\bb R$. Replacing $f$ by $\eta_k$, $2\le k\le N-1$, we
obtain that
\begin{equation*}
(\Delta_N \rho_N)(k) \;=\;0 \quad\text{for}\quad 2\le k\le N-1\;,
\end{equation*}
provided we define $\rho_N(N)$ as $\beta$. The assertion of the
lemma follows from these equations.
\end{proof}

Fix $k\in \Lambda_N \setminus\{1\}$, and place a second particle at
site $k$ at time $0$. This particle moves according to the stirring
dynamics in $\Lambda_N$ until it reaches site $1$, when it is
annihilated. This later specification is not very important in the
argument below, any other convention for the evolution of the particle
after the time it hits $1$ is fine.  Denote by $Z^k(s)$ the position
of the extra particle at time $s$ and by $d(A,j)$, $A
\subset\Lambda_N$, $j\in \Lambda_N$, the distance between $j$ and
$A$. The next lemma asserts that the process $\ms A(s)$ is extincted
before the random walk $Z^k(s)$ gets near to $\ms A(s)$ if $k\ge
\sqrt{N}$.

\begin{lemma}
\label{2-l9}
Let $\ell_N$ be a sequence such that $\ell_N\to\infty$, $\ell_N
\sqrt{N}\le N-1$. Then,
\begin{equation*}
\lim_{N\to \infty} \max_{\ell_N \sqrt{N} \le k<N}
\bb Q_N \big[ \, d(\ms A(s), Z^k(s)) =1 \text{ for some } s\ge 0 \, \big]
\;=\; 0\;.
\end{equation*}
\end{lemma}

\begin{proof}
Recall that we denote by $T$ the extinction time of the process $\ms
A(s)$.  The probability appearing in the lemma is bounded above by
\begin{equation*}
\bb Q_N \big[\, \ms A(s)\ni \ell_N \sqrt{N}/3 \text{ for some } s\ge 0
\,\big] \;+\;
\bb Q_N \big[\, \sup_{s\le T } |Z^k(s)-Z^k(0)| \ge \ell_N \sqrt{N}/3
\,\big] \;.  
\end{equation*}
By Lemma \ref{2-l2}, the first term vanishes as $N\to\infty$. Let
$m_N$ be a sequence such that $m_N\to\infty$, $m_N/\ell^2_N \to 0$. By
Lemma \ref{cl03}, the second term is bounded by
\begin{equation*}
\bb Q_N \big[\, \sup_{s\le N m_N } 
|Z^k(s)-Z^k(0)| \ge \ell_N \sqrt{N}/3 \,\big]  \;+\;  o_N(1) \;,
\end{equation*}
where $o_N(1) \to 0$ as $N\to\infty$. Since $Z^k$ evolves as a
symmetric, nearest-neighbor random walk and $m_N/\ell^2_N \to 0$, the
first term vanishes as $N\to\infty$.
\end{proof}

To prove a law of large numbers for the empirical measure under the
stationary state, we examine the correlations under the stationary
state. For $j$, $k\in \Lambda_N$, $j<k$, let
\begin{equation}
\label{2-15}
\rho_N(k) \;=\; E_{\mu_N} [\, \eta_k\,] \;, \quad
\varphi_N(j,k) \;=\; E_{\mu_N} [\,\eta_j \, \eta_k \,] 
\;-\;  \rho_N(j) \, \rho_N(k)\;.
\end{equation}

\begin{lemma}
\label{2-l6}
Let $\ell_N$ be a sequence such that $\ell_N\to\infty$, $\ell_N
\sqrt{N}\le N-1$. Then,
\begin{equation*}
\lim_{N\to\infty} \max_{\ell_N \sqrt{N} \le k< N} 
\big| \varphi_N(1,k) \big| \;= \; 0\;.
\end{equation*}
\end{lemma}

\begin{proof}
The probability $\rho_N(k)=\mu_N (\eta_k=1)$, $k\in\Lambda_N$, can be
computed by running the process $\ms A(s)$ starting from $\ms A(0)
=\{k\}$ until it is extincted, exactly as we estimated
$\rho_N(1)$. Similarly, to compute $E_{\mu_N} [\eta_1 \, \eta_k]$, we
run a process $\ms A(s)$ starting from $\ms A(0) =\{1, k\}$. In this
case, denote by $\ms A_1(s)$, $\ms A_2(s)$ the sets at time $s$ formed
by all descendants of $1$, $k$, respectively. Note that $\ms A_1(s)$
and $\ms A_2(s)$ may have a non-empty intersection. For instance, if a
particle in $\ms A_1(s)$ branches and a site $k\le p$ is occupied by a
particle in $\ms A_2(s)$.
	
To compare $E_{\mu_N}[\eta_1 \,\eta_k]$ with $E_{\mu_N} [\eta_1]\,
E_{\mu_N} [\eta_k]$, we couple a process $\ms A(s)$ starting from
$\{1,k\}$ with two independent processes $\hat{\ms A}_1(s)$, $\hat{\ms
  A}_2(s)$, starting from $\{1\}$, $\{k\}$, respectively. We say that
the coupling is successful if $\ms A_i(s) = \hat{\ms A}_i(s)$, $i=1$,
$2$, for all $s\ge 0$. In this case, the value of the occupation
variables $\eta_1$, $\eta_k$ coincide for both processes.
	
Until $d(\ms A_1(s), \ms A_2(s))=1$, it is possible to couple $\ms
A(s)$ and $\hat{\ms A}(s)$ in such a way that $\ms A_i(s) = \hat{\ms
  A}_i(s)$, $i=1$, $2$. Hence, by Lemma \ref{2-l9}, since $k\ge \ell_N
\sqrt{N}$, the coupling is successful with a probability which
converges to $1$ as $N\to\infty$.
\end{proof}

\begin{lemma}
\label{l09}
For every $\delta>0$,
\begin{equation*}
\lim_{N\to\infty} \max_{\delta N \le j<k<N}\,
\big|\, \varphi_N(j,k) \, \big| \; =\; 0\;. 
\end{equation*}
\end{lemma}

The proof of this lemma is similar to the one Lemmata \ref{l02},
\ref{l06}. As the arguments are exactly the same, we just present
the main steps.  Denote by $\widehat{\bb D}_N$ the discrete simplex
defined by
\begin{equation*}
\widehat{\bb D}_N \;=\; \{(j,k) : 2 \le j<k\le N-1\}\;,
\end{equation*}
and by $\partial\, \widehat{\bb D}_N$ its boundary:
$\partial\, \widehat{\bb D}_N = \{(1,k) : 3 \le k\le N-1\} \cup
\{(j,N) : 2 \le j\le N-2\}$. Note that the points $(1,k)$ belong to
the boundary and not to the set.  

Denote by $\ms L_N$ the generator of the symmetric, nearest-neighbor
random walk on $\widehat{\bb D}_N$ with absorption at the boundary:
For $(j,k)\in \widehat{\bb D}_N$,
\begin{gather*}
(\ms L_N \phi)(j,k) \;=\; (\bs \Delta \phi) (j,k) \;,
\quad\text{for}\quad k-j>1\;, \\
(\ms L_N \phi)(k,k+1) \;=\; (\bs \nabla^-_1 \phi)(k,k+1) 
\;+\;  (\bs \nabla^+_2 \phi)(k,k+1)
\quad\text{for}\quad  1< k < N-2 \;.
\end{gather*}
In these formulae, $\bs \nabla^\pm_i$, resp. $\bs \Delta$, represent
the discrete gradients, resp. Laplacians, introduced below equation
\eqref{g10}.

As $E_{\mu_N}[L_N \{\eta_j-\rho_N(j)\}\, \{\eta_k-\rho_N(k)\} ]=0$,
straightforward computations yield that the two-point correlation
function $\varphi_N$ introduced in \eqref{2-15} is the unique solution
of 
\begin{equation}
\label{g11b}
\begin{cases}
\vphantom{\Big\{}
(\ms L_N \psi_N) (j,k) + F_N (j,k) = 0 \;,
\;\; (j,k) \in \widehat{\bb D}_N \;, \\
\vphantom{\Big\{}
\psi_N(j,k) \,=\,  b_N(j,k) \;,
\;\; (j,k) \in \partial \, \widehat{\bb D}_N\;,
\end{cases}   
\end{equation}
where $F_N: \widehat{\bb D}_N \to \bb R$ and $b_N: \partial \,
\widehat{\bb D}_N \to \bb R$ are given by
\begin{equation*}
F_N(j,k) \;=\; -\,  [\rho_N(j+1)-\rho_N(j)]^2 \, \bs 1\{k=j+1\} \;,
\quad 
b_N(j,k) \;=\; \varphi_N(j,k) \, \bs 1\{j=1\}\;.
\end{equation*}

Denote by $\varphi^{(1)}_N$, resp. $\varphi^{(2)}_N$, the solution of
\eqref{g11b} with $b_N=0$, resp. $F_N=0$. It is clear that
$\varphi_N = \varphi^{(1)}_N + \varphi^{(2)}_N$. Let
$X_N(t) =(X^1_N(t),X^2_N(t))$ be the continuous-time Markov chain on
$\widehat{\bb D}_N \cup \partial \, \widehat{\bb D}_N$ associated to
the generator $\ms L_N$. Let $\bs P_{(j,k)}$ be the distribution of
the chain $X_N$ starting from $(j,k)$. Expectation with respect to
$\bs P_{(j,k)}$ is represented by $\bs E_{(j,k)}$.

\begin{proof}[Proof of Lemma \ref{l09}]
The piece $\varphi^{(1)}_N$ of the covariance has an explicit
expression. In view of Lemma \ref{2-l5}, for $1\le j<k\le N$,
\begin{equation*}
\varphi^{(1)}_N (j,k) \;=\; -\, \frac{[\beta - \rho_N(1)]^2}{(N-1)^2}\,
\frac{(j-1)\, (N-k)}{N-2} \;\le\; \frac{C_0}N
\end{equation*}
for some finite constant $C_0$, independent of $N$.  The piece
$\varphi^{(2)}_N$ requires a more careful analysis.

Let $H_N$ be the hitting time of the boundary $\partial\, \widehat{\bb
  D}_N$:
\begin{equation*}
H_N\;=\; \inf \big\{t\geq 0 : X_N(t) \in \partial\, \widehat{\bb D}_N \,\big\}\;.
\end{equation*}
We have that
\begin{equation*}
\varphi^{(2)}_N (j,k)  \;=\; \bs E_{(j,k)} \big[ b_N(X_N(H_N)) \big]
\;=\; \bs E_{(j,k)} \big[ \varphi_N(X_N(H_N)) \, \bs 1\{X^1_N(H_N)=1\}\, \big] \;.
\end{equation*}

Let $k_N$ be a sequence such that $k_N \ll N$. By \eqref{g16b}, for
all $\delta>0$,
\begin{equation*}
\lim_{N\to\infty} \max_{\delta N \le l<m<N} \, \bs P_{(l,m)} \big[ \, X^2_N(H_N) \le k_N \,
\big] \;=\;0\;.
\end{equation*}
Therefore, setting $k_N = \ell_N \sqrt{N}$, where
$1\ll \ell_N \ll \sqrt{N}$, by Lemma \ref{2-l6},
\begin{equation*}
\lim_{N\to\infty} \max_{\substack{(j,k) \in \widehat{\bb D}_N \\
    j>\delta N}} \big|\, \varphi^{(2)}_N (j,k) \,\big| 
\; \le\; \lim_{N\to\infty} 
\max_{\ell_N\sqrt{N} \le k <N} \big| \varphi_N(1,k) \big|  \;=\;0\;.
\end{equation*}
This proves the lemma.
\end{proof}

\begin{proof}[Proof of Theorem \ref{mt1}]
The first assertion of the theorem has been proved in Lemma
\ref{2-l5}. The proof of the second one is identical to the
proof of Theorem \ref{mt0}.
\end{proof}

\section{Speeded-up boundary conditions}
\label{sec5}

Recall that we denote by $\mu$, resp. $\mu_N$, the stationary state of
the Markov chain on $\Omega^*_p$, resp. $\Omega_{N,p}$.  Fix a smooth
profile $u: [0,1] \to (0,1)$ such that $u(0)=\rho(0)$, $u(1) = \beta$,
and let $\nu_{N,p}$ be the product measure defined by
\begin{equation*}
\nu_{N,p}(\xi,\eta) \;=\; \mu(\xi)\, \nu^N_{u}(\eta)\;, \quad \xi\in
\Omega^*_p\,,\, \eta\in\Omega_N\;,
\end{equation*}
where $\nu^N_{u}$ is the product measure on $\Omega_N$ with marginals
given by $\nu^N_{u}\{\eta_k=1\}=u(k/N)$.

Denote by $f_N$ the density of $\mu_N$ with respect to $\nu_{N,p}$,
and by $F_N: \Omega^*_p \to \bb R_+$ the density given by
\begin{equation*}
F_N(\xi) \;=\; \int_{\Omega_N}  f_N(\xi,\eta) \, \nu^N_{u}(d\eta)\;.
\end{equation*}

\begin{lemma}
\label{l10}
There exists a finite constant $C_0$ such that
\begin{equation*}
\big|\, \rho_N(0) - \rho(0)\,\big|\;\le\; C_0/\sqrt{\ell_N}
\end{equation*}
for all $N\ge 1$.
\end{lemma}

\begin{proof}
Fix a function $g: \Omega^*_p \to \bb R$. As $\mu_N$ is the stationary
state, and since $L_N g = \ell_N L_l g + L_{0,1} g$
\begin{equation*}
0\;=\; E_{\mu_N} \big[ L_N g \big] \;=\;
E_{\mu_N} \big[ \ell_N\,  L_l g + L_{0,1} g \big]\;,
\end{equation*}
so that $|\, E_{\mu_N} [ L_l g] \,| \le 2\Vert g\Vert_\infty
/\ell_N$. Since
\begin{equation*}
E_{\mu_N} [ L_l g]  \;=\; 
\int_{\Omega_{N,p}} (L_l g)(\xi) \, f_N(\xi,\eta) \, \nu_{N,p}(d\xi,d\eta) 
\;=\; \int_{\Omega^*_p} (L_l g)(\xi) \, F_N(\xi) \, \mu(d\xi) \;,
\end{equation*}
for every $g: \Omega^*_p \to \bb R$,
\begin{equation*}
\big|\, \int_{\Omega^*_p} g(\xi) \, (L^*_lF_N)(\xi) \, \mu(d\xi)\,\big| \;\le\;
2\Vert g\Vert_\infty/\ell_N\;,
\end{equation*}
where $L^*_l$ represents the adjoint of $L_l$ in $L^2(\mu)$. Since $\mu$
is the stationary state, $L^*_l$ is the generator of a irreducible
Markov chain on $\Omega^*_p$. It follows from the previous identity that
\begin{equation*}
\int_{\Omega^*_p} \, \big|\,  (L^*_lF_N)(\xi) \,\big| \, \mu(d\xi) \;\le\;
C_0/\ell_N
\end{equation*}
for some finite constant $C_0$. Hence, since $\mu(\xi)>0$ for all
$\xi\in \Omega^*_p$, $\Vert L^*_lF_N \Vert_\infty \le C_0/\ell_N$.
In particular,
\begin{equation*}
-\, \int_{\Omega^*_p} \, F_N(\xi) \,  (L^*_l F_N)(\xi) \, \mu(d\xi) 
\;\le\; (C_0/\ell_N)\,  \int_{\Omega^*_p} \, F_N(\xi) \, \mu(d\xi) 
\;\le\; C_0/\ell_N\;.
\end{equation*}
Note that the expression on the left hand side is the Dirichlet form.
Hence, by its explicit expression, $\max_{\xi, \xi'} [F_N(\xi') -
F_N(\xi)]^2 \le C_0/\ell_N$, where the maximum is carried over all
configuration pairs $\xi$, $\xi'$ such that $R(\xi, \xi') + R(\xi',
\xi)>0$, $R$ being the jump rate. In particular, as the
chain is irreducible,
\begin{equation*}
\big\Vert \, F_N - 1 \, \big\Vert_\infty \;=\;
\big\Vert \, F_N - \int_{\Omega^*_p} F_N (\xi) \, \mu(d\xi) \,
\big\Vert_\infty\;\le\; C_0/\sqrt{\ell_N}\;. 
\end{equation*}

We are now in a position to prove the lemma. One just needs to observe
that 
\begin{equation*}
\big|\, \rho_N(0) - \rho(0)\,\big|\;=\;
\Big|\, E_{\mu_N}[\eta_0] - E_{\mu}[\eta_0] \,\Big|
\;=\; \Big|\, \int_{\Omega^*_p} \xi_0 F_N(\xi) 
\, \mu(d\xi)  -  \int_{\Omega^*_p} \xi_0 
\, \mu(d\xi)\,\Big|\;,
\end{equation*}
and that this expression is bounded by $\Vert \, F_N - 1 \,
\Vert_\infty$. 
\end{proof}

Let
\begin{equation*}
\varphi_N(j,k) \;=\; E_{\mu_N} [\,\eta_j \, \eta_k \,] 
\;-\;  \rho_N(j) \, \rho_N(k)\;, \quad 
j\,,\, k \,\in\, \Lambda_{N,p} \,,\, j<k \;.
\end{equation*}

\begin{lemma}
\label{l11}
There exists a finite constant $C_0$ such that
$|\varphi_N(0,k)|\le C_0/\sqrt{\ell_N}$ for all $2\le k<N$.
\end{lemma}

\begin{proof}
The argument is similar to the one of the previous lemma.  Fix
$0<k<N$, and denote by $G_N = G^{(k)}_N : \Omega^*_p \to \bb R_+$ the
non-negative function given by
\begin{equation*}
G_N(\xi) \;=\; \int_{\Omega_N}  \eta_k\, f_N(\xi,\eta) \, \nu^N_{u}(d\eta)\;.
\end{equation*}
With this notation,
\begin{equation}
\label{g19}
E_{\mu_N} [\,\eta_0 \, \eta_k \,]  \;=\; \int_{\Omega^*_p} \xi_0\,
G_N (\xi) \, \mu(d\xi)  \;.
\end{equation}

Fix $g:\Omega^*_p \to \bb R$ and $k\ge 2$. As $k\ge 2$, 
$L_N (g \, \eta_k) = \eta_k \, L_N g + g L_N \eta_k$. Thus, 
since $\mu_N$ is the stationary state,
\begin{equation*}
0\;=\; E_{\mu_N} \big[\, L_N (g \, \eta_k) \,\big]  \;=\; 
\int_{\Omega_{N,p}}  (\ell_N \, L_l  + L_{0,1}) \, g \, \eta_k \, f_N \, d
\nu_{N,p} \;+\; E_{\mu_N} \big[\, g\, L_N \eta_k \,\big]  \;.
\end{equation*}
By definition of $G_N$ and since $|L_N \eta_k|\le 2$, $|L_{0,1} \, g|
\le 2\Vert g\Vert_\infty$, 
\begin{equation*}
\Big|\, \int_{\Omega^*_p}  (L_l  \, g)(\xi) \, G_N(\xi) \, 
\mu(d\xi) \, \Big|\;\le\; (4/\ell_N)\, \Vert g\Vert_\infty \;.
\end{equation*}

The argument presented in the proof of the previous lemma yields that
\begin{equation*}
\big\Vert \, G_N - \int_{\Omega^*_p} G_N (\xi) \, \mu(d\xi) \,
\big\Vert_\infty\;\le\; C_0/\sqrt{\ell_N}\;. 
\end{equation*}
Therefore,
\begin{equation*}
\Big|\, \int_{\Omega^*_p}  \xi_0  \,  \Big\{ G_N(\xi) 
- \int_{\Omega^*_p} G_N (\xi') \, \mu(d\xi')  \Big\} \, 
\mu(d\xi) \, \Big|\;\le\; C_0/\sqrt{\ell_N}\;. 
\end{equation*}
By definition of $G_N$ and by \eqref{g19}, the expression inside the
absolute value is equal to
\begin{equation*}
E_{\mu_N} [\,\eta_0 \, \eta_k \,] \;-\; \rho(0)\, \rho_N(k)\;.
\end{equation*}
The assertion of the lemma follows from the penultimate displayed
equation and from Lemma \ref{l10}.
\end{proof}

\begin{proof}[Proof of Theorem \ref{mt2}]
The first assertion of the theorem is the content of Lemma \ref{l10}. 
The proof of Lemma \ref{l09} [with $\widehat{\bb D}_N$ defined as
$\widehat{\bb D}_N \;=\; \{(j,k) : 1 \le j<k\le N-1\}$] yields that
for every $\delta>0$,
\begin{equation*}
\lim_{N\to\infty} \max_{\delta N \le j<k<N}\,
\big|\, \varphi_N(j,k) \, \big| \; =\; 0\;. 
\end{equation*}
A Schwarz inequality, as in the proof of Theorem \ref{mt0}, completes
the argument because
$\rho_N(k) = (k/N)\, \beta + [1-(k/N)]\, \rho_N(0)$, $1\le k\le N$.
\end{proof}

\smallskip\noindent{\bf Acknowledgments.} We thank H. Spohn for
suggesting the problem and S. Grosskinsky for fruitful discussions.
C. Landim has been partially supported by FAPERJ CNE
E-26/201.207/2014, by CNPq Bolsa de Produtividade em Pesquisa PQ
303538/2014-7, and by ANR-15-CE40-0020-01 LSD of the French National
Research Agency.


\begin{thebibliography}{99}

\bibitem{bdgjl02} L. Bertini, A. De Sole, D. Gabrielli,
  G. Jona-Lasinio, and C. Landim, , Macroscopic fluctuation theory for
  stationary non- equilibrium states, J. Stat. Phys. {\bf 107}, 635
  (2002).

\bibitem{bdgjl15} L. Bertini, A. De Sole, D. Gabrielli,
  G. Jona-Lasinio, and C. Landim, Macroscopic fluctuation
  theory. Rev. Mod. Phys. {\bf 87}, 593--636, (2015).

\bibitem{d07} B. Derrida: Non-equilibrium steady states: Fluctuations
  and large deviations of the density and of the current.
  J. Stat. Mech.  Theory Exp., P07023 (2007).

\bibitem{dls02} B. Derrida, J. L. Lebowitz, and E. R. Speer.  Large
  deviation of the density profile in the steady state of the open
  symmetric simple exclusion process. J. Stat. Phys. {\bf 107}, 599
  (2002).

\bibitem{e17} C. Erignoux: Hydrodynamic limit of boundary exclusion
  processes with nonreversible boundary dynamics. preprint
  arXiv:1712.04877, (2017).

\bibitem{els90} G. Eyink, J. Lebowitz, H. Spohn: Hydrodynamics of
  stationary non-equilibrium states for some stochastic lattice gas
  models. Comm. Math. Phys. {\bf 132} 253--283 (1990).

\bibitem{fri75} A.  Friedman: {\it Stochastic differential equations
    and applications.}  Volume 1, Academic Press, New York, 1975.

\bibitem{klo95} C. Kipnis, C. Landim and S. Olla: Macroscopic
  properties of a stationary non–equilibrium distribution for a
  non–gradient interacting particle system, Ann. Inst. H. Poincar\'e,
  Prob. et Stat. {\bf 31}, 191–221 (1995).

\bibitem{lov98} C. Landim, S. Olla, S. Volchan: Driven tracer particle
  in one-dimensional symmetric simple
  exclusion. Comm. Math. Phys. {\bf 192}, 287--307 (1998).

\bibitem{law91} G. F. Lawler: {\it Intersections of Random Walks}.
  Modern Birkh\"auser Classics, Birkh\"auser Basel, 1991.

\bibitem{o31} L. Onsager: Reciprocal relations in irreversible
  processes.  I, II Phys. Rev. {\bf 37}, 405 and {\bf 38}, 2265 (1931)

\bibitem{om53} L. Onsager and S. Machlup: Fluctuations and
  irreversible processes. Phys. Rev. {\bf 91}, 1505 (1953)

\bibitem{s09} N. Sonigo: Semi-infinite TASEP with a complex boundary
  mechanism. J. Stat. Phys. {\bf 136} 1069-1094 (2009)

\end{thebibliography}
\end{document}